\documentclass[12pt]{amsart}
\usepackage{amsmath, amsthm, amssymb}
\usepackage[utf8]{inputenc}
\usepackage[margin=2in, top=20mm, footskip=.30in]{geometry}
\usepackage[english]{babel}
\usepackage{amsmath,amsthm,amssymb,color,fullpage,enumitem,tikz,float,subcaption,algpseudocode,algorithm,bbm}
\usepackage[indent=10pt]{parskip}
\usepackage{authblk}
\usepackage[colorlinks=true,linkcolor=blue]{hyperref}
\usepackage{amsaddr}
\usepackage{xcolor}
\usepackage{soul}

\newcommand{\CC}{\mathbb{C}}

\newcommand{\EE}{\mathbb{E}}

\newcommand{\sss}{\mathbf{s}}
\newcommand{\yy}{{\mathbf{y}}}
\newcommand{\ii}{{\mathbbm{i}}}
\newcommand{\xx}{{\mathbf{x}}}
\newcommand{\ttt}{{\mathbf{t}}}

\newcommand{\RR}{\mathbb{R}}

\newcommand{\indep}{\perp \!\!\! \perp}

\newcommand{\cV}{\mathcal{V}}

\newcommand\scalemath[2]{\scalebox{#1}{\mbox{\ensuremath{\displaystyle #2}}}}

\newcommand{\var}{{\operatorname{var}}}

\newtheorem{thm}{Theorem}[section]
\newtheorem{rmk}{Remark}[section]
\newtheorem{lem}{Lemma}[section]
\newtheorem{cor}{Corollary}[section]
\newtheorem{prop}{Proposition}[section]
\newtheorem{conj}{Conjecture}[section]

\newtheorem{defn}{Definition}[section]
\newtheorem{exmp}{Example}[section]

\newcommand{\ti}[1]{\textit{#1}}

\newcommand{\cum}[3]{\kappa_{#1}(#2)_{#3}}
\newcommand{\gencum}{\cum{r}{\xx}{i_1\ldots i_r}}

\title{Cumulant Tensors in Partitioned Independent Component Analysis}
\author{Marina Garrote-L\'opez AND Monroe Stephenson}
\address{Max Planck Institute for Mathematics in the Sciences}
\email{marina.garrote@mis.mpg.de; monroe.stephenson@mis.mpg.de}

\begin{document}

\maketitle

\begin{abstract} 
In this work, we explore Partitioned Independent Component Analysis (PICA), an extension of the well-established Independent Component Analysis (ICA) framework. Traditionally, ICA focuses on extracting a vector of independent source signals from a linear combination of them defined by a mixing matrix. We aim to provide a comprehensive understanding of the identifiability of this mixing matrix in ICA. Significant to our investigation, recent developments by Mesters and Zwiernik relax these strict independence requirements, studying the identifiability of the mixing matrix from zero restrictions on cumulant tensors. In this paper, we assume alternative independence conditions, in particular, the PICA case, where only partitions of the sources are mutually independent. We study this case from an algebraic perspective, and our primary result generalizes previous results on the identifiability of the mixing matrix.
 
\end{abstract}

\section{Introduction}

Independent Component Analysis (ICA) has long been a powerful tool for blind source separation, transforming observed random vectors into statistically independent components \cite{bookICA, Comon2010, aapo1, icaintro}. Traditionally, ICA assumes mutual independence among the elements of the hidden vector, leading to identifiable solutions for the mixing matrix up to sign permutations or orthogonal transformations, as established by Comon's work \cite{comon}. 

In more detail, consider the scenario where we observe $d$ random variables $y_1,\ldots,y_d$, modeled as a linear combination of $d$ random variables $s_1,\ldots,s_d$, represented by the equation $$\yy = A\sss,$$ where $\yy\in\RR^d$ is observed, $A\in\RR^{d\times d}$ is an unknown invertible matrix, and $\sss\in\RR^d$ is hidden. The process of recovering $A$ is termed Component Analysis, and when the random variables $s_i$ are mutually independent, it is referred to as Independent Component Analysis (ICA). Common assumptions include mean zero for both $\sss$ and $\yy$, and unit variance for each $s_i$ (i.e., $\mathbb{E}(s_i^2)=1$). This leads to $\text{cov}(\yy) = AA^T$ and restricts the identifiability of $A$ to the set $\{QA\mid Q\in O(d)\}$, where $O(d)$ is the group of orthogonal matrices.

In the context of ICA, the assumption that the components of the vector $\sss$ are independent holds significant relevance, supported by Comon's theorem, since this ensures the identifiability of $A$ up to signed permutations. However, this independence requirement may prove overly restrictive in certain applications \cite{Matteson2017, aapoTopographic}. Cardoso \cite{Cardoso} pioneered an extension known as Independence Subspace Analysis (ISA) \cite{towards, ICASS, aapoEmergence} or multidimensional Independence Component Analysis \cite{Theis2004}. The fundamental concept behind ISA is to estimate an invertible matrix $A$ for a given observation $\yy\in\RR^d$ such that $\yy = A(\sss_1,\ldots, \sss_m )^T$, where $\sss_i$ represents mutually independent random vectors. 
While the traditional ICA scenario assumes one-dimensional vectors, ISA introduces flexibility in dimensionality. However, if mutual independence is the sole constraint imposed on $\sss$ without dimensionality restrictions, meaningful results from PICA become unattainable, necessitating additional constraints. {In general, ISA does not assume extra independence among the components of each random vector $\sss_i$. In this work, we introduce the concept of {Partitioned Independence Component Analysis} (PICA) where the vectors $\sss_i$ may satisfy additional conditions of independence.}

Recent developments by Mesters and Zwiernik \cite{nonica} have delved into scenarios where independence among variables $\sss$ are not required but some restrictions on certain cumulant tensors of $\sss$ are imposed. The authors explore minimal assumptions that relax the strict independence requirement while still guaranteeing the identifiability of the matrix $A$ up to signed permutations or within a finite subset of orthogonal matrices.
More specifically, they study zero restrictions on higher-order moments and cumulant tensors that arise from specific models or alternative relaxations of independence that assure the identifiability of $A$. They focus specifically on diagonal tensors and reflectionally invariant tensors motivated by mean independence.

In this article, we pursue two primary objectives. First, we aim to clearly summarize and present the essential results necessary to prove Comon's theorem, addressing the identifiability of $A$ in the ICA context. Additionally, we extend the analysis to the {ISA} 
case following the approach of Theis \cite{Theis2004}.
Second, we follow the approach of Mesters and Zwiernik in \cite{nonica} to analyze PICA from an algebraic perspective. This involves considering cumulant tensors with a specific block structure, resulting from partitioned independence in $\sss$, and analyzing the subsets within which $A$ can be identified.
{Our main results are Theorem \ref{thm:generalThmAlgPICA} and Corollary \ref{cor:reflInvTensor}, which generalizes the identifiability of $A$ for PICA case. Moreover, we introduce the independence of $\sss$ induced by a graph $G$ and give partial results on the identifiability of $A$ for these cases.
We anticipate that this work will serve as an introduction and inspiration for employing more algebraic tools to address open questions.}

This paper is organized as follows. In Section \ref{sec:ICA} we formally introduce ICA, explore the necessary restrictions one may make in the model and present Comon's result on the identifiability of $A$. In Section \ref{sec:cum}, we introduce moment and cumulant tensors and present their main properties. We also discuss the zero restrictions on the cumulants of $\sss$ given by certain dependencies on the entries of $\sss$ in Section \ref{sec:cumIndep}. Following, Section \ref{sec:pica} is devoted to the PICA model, which is formally introduced in Section \ref{sec:PICAmodel}. We continue by showing an algebraic perspective for component analysis in Section \ref{sec:connectionAlgStats}, and in the remaining sections we show how cumulant tensors play a role in the case of partitioned independent random vectors and alternative independence. Finally Section \ref{sec:statisticalPICA} is dedicated to the identifiability of $A$ for the PICA model.

\section{Independent Component Analysis}\label{sec:ICA}

In this section, we review the basic concepts of the classical independent component analysis (ICA). We introduce the mathematical formulation of ICA, which is given as a statistical estimation problem and we discuss some preliminary assumptions that may be made and under what assumptions this model can be estimated. 

\subsection{The ICA model}
In general, measurements cannot be isolated from noise. 
For example, consider a live music recording, this has the sounds of drums, guitar, and other instruments as well as the audience's applause. Therefore, it is challenging to have a clean recording of one instrument due to noise (audience's applause) and the other independent signals generated from the other instruments (guitar, bass, ...). We can define the measurements as a combination of multiple independent sources. The topic of separating these mixed signals is called {blind source separation}. The independent component analysis technique is one of the most used algorithms to solve the problem of extracting the original signals from the mixtures of signals. 

\ti{Source signals}, $\sss = (s_1,\hdots, s_d)$, vary over time and may be represented as time series where each step represents the amplitude of the signal $s_i$ at a certain time step. For us, $\sss\in\RR^d$ will be a hidden random vector. The source signals are assumed to be linearly mixed as $y_i = \sum_ja_{ij}s_j$. We will call the entries $y_i$ of $\yy\in\RR^d$, the \ti{mixture signals}. 
That is, \begin{equation}\label{eq:icastatement}
    \mathbf{y}=A\mathbf{s},
\end{equation} where 
$A\in \RR^{d\times d}$ is an unknown invertible matrix, called the \ti{mixing coefficient matrix}. The task of ICA is to recover $\sss$ given only the observations $\yy$.

Typically, ICA models make
certain assumptions and restrictions that allow for the identifiability of the mixing matrix $A$. We introduce them here and we will see their importance in Section \ref{sec:comon}.
It is assumed that the sources $s_i$ are mutually independent and that there is at most one source $s_i$ that has a Gaussian distribution. 
%
Under these assumptions, the ICA model is identifiable, that is, $A$ can be estimated up to permutation and sign changes of the rows, as we will see in Section \ref{sec:comon}.

It is also worth noting that some ambiguities or features will never be identifiable, since both $A$ and $\sss$ are unknown.
First, it is clear that we cannot identify the variances of the sources $s_i$. That is, any scalar multiplier $b_i$ in one of the variables $s_i$, that is $b_is_i$, can be cancelled by dividing the corresponding column $a_i$ of $A$ by the same scalar $b_i$, and giving the same vector $\yy$. As a consequence, we fix the variances of each $s_i$ and assume that $\EE(s_i^2)=1.$ 

The other feature that can not be identified is the order of the sources $s_i$. That is, consider a permutation matrix $P$, then one can consider the new random vector $P\sss$ and the matrix $AP^{-1}$ that gives rise to the same ICA model $(AP^{-1})(P\sss) = A\sss=\yy.$

Two preprocessing steps can be done on the data $\yy$ that simplify the recovery of the matrix $A$: \textit{centering} and \textit{whitening}.

\newpage
\par \textbf{Centering:} 
Centering the mixture variables $y_i$ means transforming them such that they have mean zero. That is, we can simply define
new mixture signals $$\yy'=\yy-\EE(\yy). 
$$
This also implies that the independent components $s_i$ have mean zero,
since $\EE(\yy) = A\EE(\sss).$
After estimating $A$ we may add the mean vector back to the independent components.
Note that this does not affect the estimation of $A$, as it remains the same after centering $\yy$. 
Therefore, from now on we will assume that both the mixing and the source variables $y_i$ and $s_i$ have mean zero.

\textbf{Whitening}: 
We say that a random vector $\yy$ is \ti{white} if it is uncorrelated with mean zero, that is, $\EE(\yy \yy^T)=\operatorname{Id}$. Then, by \ti{whitening} the signal $\yy$ we mean to linearly transform it into a new random vector $\yy'$ with identity covariance. 

One way to whiten $\yy$ is by considering the eigendecomposition of its covariance matrix $\Sigma = \EE [\yy \yy^T]$. Let $\Lambda$ be a diagonal matrix with the eigenvalues of $\Sigma$ and $V$ the orthogonal matrix of eigenvectors.
Then we can define the matrix $Z=V\Lambda^{-1/2}V^T$, and the random vector $\yy' = Z\yy$. It is straightforward to see that $\EE [\yy' \yy'^T] = Id.$ 
This process of decorrelating the signals $\yy$ is known as the {Principal Component Analysis} technique.

Finally, observe that if we consider the new random vector $\yy'$, the problem \eqref{eq:icastatement} becomes $$\yy' = Z\yy = ZA\sss = A'\sss, $$ where $A'$ is an orthogonal matrix. This means we can restrict our search for the new mixing matrix $A'$ to the group of orthogonal matrices.

To summarize, from now on, unless noted otherwise we will consider the following assumptions on $\yy=A\sss$. First, we assume that the variables $s_i$ are mutually independent, have unitary variance $\EE[s_i^2]=1$ and there is at most one Gaussian. We will also assume that both source and mixture signals $\sss$ and $\yy$ are centered, that is they have mean zero.
Finally, if $\yy$ is not assumed to be white, we can restrict the search of $A$ to the set $\{QA|Q\in O(n)\}.$ Otherwise, if $\yy$ is whitened, then $A\in O(n)$.

\par When performing ICA, the goal is to compute the \textit{unmixing matrix} $W:=A^{-1}$ such that \begin{equation}\label{eq:inverseICA}
    \sss = W\yy.
\end{equation}

{Even if it not the goal of this paper, it is worth pointing out that there exist different approaches to estimate the matrix $W$.} 
For instance, one first approach is based on non-Gaussianity measured by negentropy and kurtosis to find independent components that maximize the non-Gaussianity \cite{nongauss1,nongauss2}. As an alternative, ICA can be obtained by minimizing the mutual information \cite{mutualinfo1,mutualinfo2}. The independent components can also be estimated by using maximum likelihood estimation \cite{maxliklihood}. All of these methods search for the unmixing matrix $W$ and then project the whitened data onto the matrix to extract the independent signals. We refer the reader to \cite{icaintro} for a brief overview of these methods.

\subsection{Identifiability of ICA}\label{sec:comon}
In this subsection, we discuss the identifiability of the matrix $A$ from \eqref{eq:icastatement}. Our goal is to review and prove the classical \textit{Comon's Theorem} \cite{comon} which states that the matrix $A$ can be
identified up to permutation and sign transformations of its rows.

\begin{thm}[Comon \cite{comon}]\label{thm:comons}
    Let $\sss \in \RR^d$ be a vector with mutually independent components, of which at most one is Gaussian. Let $A$ be an orthogonal $d\times d$ matrix and $\yy$ the white vector $\yy = A \sss.$ The following are equivalent:
    \begin{enumerate}
        \item The components $y_i$ are pairwise independent.
        \item The components $y_i$ are mutually independent.
        \item $A$ is a signed permutation matrix.
    \end{enumerate}
\end{thm}
Comon's Theorem relies primarily on 
Corollary \ref{cor:darmois}. In order to be able to state this corollary, we need to prove Cram\'er's Lemma \cite{cramer_1970} (Lemma \ref{lem:Cramer}), Marcinkiewicz's Lemma \cite{Marcinkiewicz1939} (Lemma \ref{lem:Marc}) and Darmois \cite{darmois} and Skitovich \cite{skitovich} Theorem (Theorem \ref{thm:darm}). We leave then proof of Comon's Theorem for the end of this subsection. 
Throughout this section, we will mostly follow the proofs given in \cite{normal}.

We start by giving some preliminary definitions and lemmas.

\begin{defn}
    The characteristic function of a random variable $x$ is defined as $\phi_x(t) = \EE(e^{\mathbbm{i}t x}).$ When $\phi_x$ is complex-valued we use $\phi_x(z).$
\end{defn}

\begin{defn}
 Suppose that $f$ is an analytic function on some open subset $U$ of $\CC$. Let $F(z)$ be an analytic function on $V\supseteq U$ such that $f(z)=F(z)$ for all $z\in U$, then $F$ is an \textit{analytic extension of $f.$}
\end{defn}

\begin{rmk}[\cite{normal}]
    There are two essential facts for characteristic functions that will be needed later:
\begin{enumerate}
    \item The characteristic function of the standard Gaussian random variable $x$ is $\phi_x(z)=e^{-z^2/2}$ for all complex $z\in \CC.$ In other words, the analytic extension for a Gaussian random variable is $e^{-z^2/2}.$ 
    \item If a characteristic function can be represented in the form $\phi_x(t)=e^{at^2+bt+c}$ for some constants $a,b,c$, then $\phi_x$ corresponds to a Gaussian random variable.
\end{enumerate}
\end{rmk}

\begin{lem}
    If $x$ is a random variable such that $\EE \left(e^{\lambda x^2}\right)< \infty$ for some $\lambda >0,$ and the analytical extension $\phi_x(z)$ of the characteristic function of $x$ satisfies $\phi_x(z)\neq 0$ for all $z\in \CC$, then $x$ is Gaussian. \label{lem:finiteness}
    \begin{proof}
        We know $f(z)=\log \phi_x(z)$ is well defined and entire. Assuming $z=x+\ii y$, then $$\operatorname{Re}f(z)=\log |\phi(z)|\le \log  \EE(|e^{\ii (x+\ii y)x}|) \le  \log \left(\EE e^{|yx|}\right).$$ It is known that $\EE \left( e^{tx}\right) \le C e^{t^2/2\lambda}$, where $C=\EE \left(e^{\lambda x^2}\right)$, for all real $t$ [Problem 1.4, \cite{normal}]. Then, since $\lambda x^2 +t^2/2\lambda \ge t x$, we have $$\EE(e^{tx})\le  C e^{t^2/2\lambda}= \EE \left(e^{(\lambda x^2+t^2/2\lambda)}\right).$$ By using the two equations we have derived we find $\operatorname{Re}f(z)\le c +\frac{y^2}{2\lambda}$ where $c \in \RR.$ Hence, using the Liouville Theorem, we get that $f(z)$ is a quadratic polynomial in the variable $z$, so $x$ must be Gaussian. 
    \end{proof} 
\end{lem}
\begin{lem}
    If a random variable $x$ has finite exponential moment $\EE \left(e^{a|x|}\right)< \infty,$ where $a>0$, then its characteristic function is analytic in the strip $-a < \operatorname{Im}(z) <a.$ \label{lem:charAnalytic}
    \begin{proof}
        The analytic extension is explicitly $\phi_x(z) =\EE \left(e^{\ii zx}\right).$ We only need to check that $\phi_x(s)$ is differentiable in the strip $-a < \operatorname{Im} (z)<a.$ To do this we just write its series expansion (using Fubini to switch integration and summation): $$\phi_x(z) =\sum_{i=0}^\infty \ii^i \frac{\EE(x^i)z^i}{i!}$$
        which is absolutely convergent for all $-a < \operatorname{Im}(z) <a.$
    \end{proof}
\end{lem}

The following two lemmas are essential to prove Comon's theorem.

\begin{lem}[Cram\'er \cite{cramer_1970}]
    If $x_1, \hdots, x_d$ are independent random variables, and if $x=\sum_{i=1}^d a_i x_i$ is Gaussian, then all the variables $x_i$ for which $a_i\neq 0$ are Gaussian. \label{lem:Cramer}
    \begin{proof}
        Without loss of generality let $\EE x_i=0$. In this case, it is known that $\EE \left(e^{ax^2_i}\right)<\infty$ for $i\in [d]$ (Theorem 2.3.1 \cite{normal}). Thus, by Lemma \ref{lem:charAnalytic}, the corresponding characteristic functions $\phi_{x_i}$ are analytic. By the uniqueness of the analytic extension, $\prod_{i=1}^d \phi_{x_i}(z)= e^{-z^2/2}$ for all $z\in \CC$. Hence $\phi_{x_i}(z)\neq 0$ for all $z\in \CC$ and all $i\in [d]$, thus by Lemma \ref{lem:finiteness} all the characteristic functions correspond to Gaussian distributions.
    \end{proof}
\end{lem}
\begin{lem}[Marcinkiewicz \cite{Marcinkiewicz1939}]\label{lem:Marc}
    The function $\phi(t)=e^{P(t)}$ where $P(t)$ is a polynomial of degree $n$, can be a characteristic function only if $n\le 2.$ Then $\phi$ corresponds to the characteristic function of a Gaussian distribution. 
    \begin{proof}
        First observe that the function $\phi(z)=e^{P(z)}$ for $z\in \CC$ defines an analytical extension of $\phi(t).$ Thus by Lemma \ref{lem:charAnalytic}, $\phi_x(z)= \EE \left(e^{\ii zx}\right)$ for $z\in \CC.$ By Cram\'er's Lemma \ref{lem:Cramer}, it suffices to show that $\phi_x(z)\phi_{x}(-z)$ corresponds to a Gaussian distribution. We have that, $$\phi_x(z)\phi_{x'}(-z)=\EE \left(e^{\ii z(x-x')}\right),$$
        where $x'$ is an independent copy of $x$. If we show that $x-x'$ is Gaussian then we can use Lemma \ref{lem:Cramer} to infer that $x$ is Gaussian. That is, consider, $$\phi_x(z)\phi_{x'}(-z)=e^{R(z)},$$
        where the polynomial $R(z)=P(z)+P(-z)$ has only even terms, i.e. $R(z) = \sum_{i=0}^{\left \lfloor{n/2}\right \rfloor } p_{2i} z^{2i}$ where $p_{2i}$ are the coefficients.
        
         We know that $\phi_{x}(t)\phi_{x'}(-t) =|\phi(t)|^2$ since characteristic functions are Hermitian and it is a real number for some real $t\in \RR$. Moreover, $R(t) = \log |\phi_x(t)|^2$ is real too. It follows that the coefficients $p_{2i}$ of the polynomial $R(t)$ are real. Furthermore we will show that $p_n<0$ since $|\phi_x(t)|^2<1$ for all real $t.$ Suppose that the highest coefficient $p_n$ is positive. Then for $t>1$ we have $$e^{P(t)} \ge e^{p_n t^{2n}-nBt^{2n-2}}= e^{t^2n\left(p_n -nB/t^2\right)}\to \infty \text{ as }t\to \infty$$
        where $B$ is the largest coefficient (in value) of $R(t).$ Hence, it follows that $p_n$ is negative. For simplicity, we denote $p_n=-\gamma^2$. 
        
        \par Let $w(N)=\sqrt[2n]{N} e^{i\pi /2n}$. In this case, $w^{2n}= -N.$ For $N>1$ we have \begin{align*}
            |\phi_x(w)\phi_{x'}(-w)|&=\left|e^{R(w)}\right| = \left|\exp \left(\sum_{i=0}^{\left \lfloor{n/2}\right \rfloor -1} p_i w^{2i}+\gamma^2 N\right)\right|\ge\exp \left( \gamma^2 N -B \sum_{i=0}^{\left \lfloor{n/2}\right \rfloor -1}|w^{2i}|\right) \\
            &\ge \exp\left( \gamma^2 N- nBN^{1-1/n}\right)=\exp\left(N\left(\gamma^2 - \frac{nB}{N^{1/n}}\right)\right). 
        \end{align*}
        Hence, we have $|\phi_x(w)\phi_{x'}(-w)|\ge e^{N(\gamma^2-\epsilon_N)}$ for all large enough real $N$ where $\epsilon_N\to 0$ as $N\to \infty.$
        
        \par Now we will look at an inequality for the other side of $|\phi_x(w)\phi_{x'}(-w)|.$ Using the explicit representation of $\phi_x(w)$, Jensen's inequality, and the independence of $x$, $x'$ we get 
        \begin{align*}
            |\phi_x(w)\phi_{x'}(-w)|&=\left|\EE\left( e^{\ii w(x-x')}\right)\right|\le \EE\left(e^{-\mathfrak{s}\sqrt[2n]{N}(x-x')}\right)= \left|\phi_{x}\left(\ii \mathfrak{s}\sqrt[2n]{N}\right)\phi_{x'}\left(-\ii \mathfrak{s}\sqrt[2n]{N}\right)\right|
        \end{align*}
        where $\mathfrak{s}= \sin (\pi/2n)$ so $\operatorname{Im}(w)= \mathfrak{s}\sqrt[2n]{N}$. Thus $|\phi_x(w)\phi_{x'}(-w)|\le |e^{R(\ii \mathfrak{s} \sqrt[2n]{N})}|.$
        
         Note that since $R$ only has even terms then $R(\ii \mathfrak{s} \sqrt[2n]{N}) \in \RR.$ For $N>1$, $$R(\ii \mathfrak{s} \sqrt[2n]{N})= -(-1)^n \gamma^2 N \mathfrak{s}^{2n} +\sum_{i=0}^{n-1} (-1)^i p_i \mathfrak{s}^{2i} N^{2i/n}\le \gamma^2 N \mathfrak{s}^{2n}+nBN^{1-1/n}= N\left(\gamma^2 \mathfrak{s}^{2n}+\frac{nB}{N^{1/n}}\right).$$
        It follows that $$e^{N(\gamma^2 -\epsilon_N)} \le |\phi_x(w)\phi_{x'}(-w)|\le e^{N(\gamma^2 \mathfrak{s}^{2n}+\epsilon_N)}$$
        where $\epsilon_N\to 0.$ As $N\to \infty$ this yields a contradiction unless $\mathfrak{s}^{2n}\ge 1$, i.e. $\sin (\pi/2n) = \pm 1.$ We only encounter this possibility when $n=2$, so $R(t)=p_0 +p_1t^2$ is of degree $2.$ Since $R(0)=0$ and $p_1 = -\gamma^2$ we have $R(t) = -(\gamma t)^2$ for all $t,$ thus $\phi_x(t)\phi_{x'}(-t) = e^{-(\gamma t)^2}$ is Gaussian and so by Lemma \ref{lem:Cramer}, $\phi_x(t)$ corresponds to a Gaussian distribution.
    \end{proof}
\end{lem}

Furthermore, the following corollary of Lemma \ref{lem:Marc} will be relevant to our discussion. 
\begin{cor}
    There does not exist any polynomial cumulant generating functions of degree greater than $2.$
\end{cor}
We need one final lemma before proving Theorem \ref{thm:darm}.
\begin{lem} \label{lem:darLem}
    Suppose that the functions $f_1,\hdots, f_n$ are differentiable of all orders and we have that $$\sum_{i=1}^n f_i(\alpha_ix+\beta_iy) = g(x)+h(y) \text{ for all }x,y$$
    where $\alpha_i$, $\beta_i$ are non-zero constants such that for $i\neq j,$
    $$\alpha_i\beta_j-\alpha_j\beta_i\neq 0.$$
    Then, all the functions $f_i$ are polynomials with degree at most $n.$
    \begin{proof}
        We can automatically see that $g$ and $h$ will be differentiable for any order as well. Suppose that there are small variations $\delta_1^{(1)}$ and $\delta_2^{(1)}$ such that $\alpha_n \delta_1^{(1)}+\beta_n \delta_2^{(1)}=0$ then, we can redefine $x$ and $y$ as 
 %
        $$x \leftarrow x+ \delta_1^{(1)} \qquad
        y \leftarrow y+ \delta_2^{(1)}$$
        such that $\alpha_n x+\beta_n y$ remains constant.
        Simultaneously, the arguments of all the other $f_i 's$ have changed by a small value $\epsilon_i^{(1)}$ which is non-zero. Hence, subtracting the new equation from the original we get $$\sum_{i=1}^{n-1}\Delta_{\epsilon_i^{(1)}}f_i(\alpha_ix+\beta_iy) =g_1(x)+h_1(y) \text{ for all }x,y$$
        where $\Delta_kf(x) = f(x+k)-f(x)$ is the finite difference of $f(x)$.
        \par We can see that the last equation is similar to the original one with the absence of $f_n.$ Repeating this process we get, $$\Delta_{\epsilon_1^{n-1}} \hdots \Delta_{\epsilon_1^{1}} f_1(\alpha_1x+\beta_1y) =g_{n-1}(x)+h_{n-1}(y) \text{ for all }x,y. $$
        Finally, we do this two more times (once for small variation only in $x$ and once for small variation only in $y$) and we arrive at $$\Delta_{\epsilon_1^{n+1}} \hdots \Delta_{\epsilon_1^{1}} f_1(\alpha_1x+\beta_1y) =0 \text{ for all }x,y. $$
        Specifically, the $(n+1)$-th order difference of the function $f_1$, and thus its $(n+1)$-th order derivative, are zero. Hence $f_1$ is a polynomial of degree at most $n.$ This process can be repeated for all of the other functions $f_i$'s.
    \end{proof}
\end{lem}

The following theorem states that if two linear combinations of independent random variables are independent, then the variables are Gaussian. This will allow us to state the corollary needed to prove Comon's Theorem.
\begin{thm}[Darmois \cite{darmois}, Skitovich\cite{skitovich}] \label{thm:darm}
    Let $y_1$ and $y_2$ be two random variables defined as $$y_1= \sum_{i=1}^n \alpha_i x_i\quad \quad y_2= \sum_{i=1}^n \beta_i x_i$$
    where $x_i$ are independent random variables. Then if $y_1$ and $y_2$ are independent, all variables $x_i$ for which $\alpha_i\beta_j\neq 0$ are Gaussian. 
    \begin{proof}
        Without loss of generality, we can assume $\alpha_i \beta_j -\alpha_j \beta_i\neq 0$ for all $i\neq j$ (if this does not hold then we can add $y_1$ and $y_2$ and use Lemma \ref{lem:Cramer} to prove Gaussianity of both). 
        \par Now consider $$\phi_{y_1 y_2}(s,t)= \EE\left(e^{\ii(sy_1+t y_2)}\right)= \EE \left(e^{\ii\sum_{i}(\alpha_is+\beta_i t)x_i}\right)=\prod_{i=1}^n \phi_{x_i}(\alpha_i s+ \beta_i t).$$
        The last equation comes from the independence of the $x_i$'s. We also know $y_1$ and $y_2$ are independent so $$\phi_{y_1y_2}(s,t) = \phi_{y_1}(s)\phi_{y_2}(t),$$
        so it follows $$\prod_{i=1}^n \phi_{x_i}(\alpha_i s+ \beta_i t)=\phi_{y_1}(s)\phi_{y_2}(t).$$
        Let $\psi_{x_i}(s)=\log \phi_{x_i}(s)$, then taking the logarithm of the previous expression we get
        $$\sum_{i=1}^n \psi_{x_i}(\alpha_i s+ \beta_i t)=\psi_{y_1}(s)+\psi_{y_2}(t).$$
        
        \par Move the terms of the left side for which $\alpha_i\beta_i=0$ to the right side, then by Lemma \ref{lem:darLem} $\psi_{x_i}$ must be a polynomial for those $i$ such that $\alpha_i\beta_i\neq0$. We conclude then that those $x_i$, with $\alpha_i\beta_i\neq 0$, must be Gaussian random variables using Lemma \ref{lem:Marc}.
    \end{proof}
\end{thm}
\begin{cor} \label{cor:darmois}
    Let $\yy\in\RR^d$ and $\sss\in\RR^d$ be two random vectors such that $\yy= A\sss$ where $A$ is a square matrix. Furthermore, suppose that $\sss$ has mutually independent components and that $\yy$ has pairwise independent components. If $A$ has two non-zero entries in the same column $j$, then $s_j$ is either Gaussian or constant.
\end{cor}
Now we may present the proof of Comon's Theorem, that is, Theorem \ref{thm:comons}.
\begin{proof}[Proof of Theorem \ref{thm:comons}]
    We can see that $(3)\implies (2)$ since that would just mean $\yy$ is some permutation of $\sss$ up to sign, so $\yy$ would be mutually independent. $(2)\implies (1)$ is by definition. 
    \par It remains to prove $(1)\implies (3)$. Hence, we may assume that $\yy$ has pairwise independent components. Suppose for sake of contradiction that $A$ is not a sign permutation. We know that $A$ is orthogonal, thus it necessarily has two non-zero entries in at least two different columns. Using Corollary \ref{cor:darmois} twice we get that $\yy$ has at least two Gaussian components, which is a contradiction with our assumption.
\end{proof}

\section{Cumulant Tensors} \label{sec:cum}
In this section, we give a brief introduction to cumulant tensors. The reader is referred to \cite{tensorMethods,Jacod_Protter_2003,Zwiernik2015} for more details.

\subsection{Symmetric tensors}

{In this section we introduce symmetric tensors and state some properties that will be needed later in this work.}

A symmetric tensor $T\in \RR^{d\times \hdots \times d}$ is a tensor that is invariant under arbitrary permutation of its indices:
$T_{i_1\ldots i_r}=T_{i_{\sigma(1)},i_{\sigma (2)},\ldots ,i_{\sigma (r)}}$ for any permutation $\sigma$ of $[r]$.
We will denote the space of real symmetric $d\times \hdots \times d$ tensors of order $r$ by $S^r(\RR^d)\subset \RR^{d\times \hdots \times d}$. The unique entries $T\in S^r(\RR^d)$ are denoted as $T_{i_1,\hdots, i_r}$ where $1\le i_1\le \hdots \le i_r\le d.$ 

Given a matrix $A\in \RR^{d\times d}$ and a symmetric tensor $T\in S^r(\RR^d)$ we can define the standard \emph{multilinear operation} as: $$(A\bullet T)_{i_1,\hdots, i_r}=\sum_{j_1,\hdots, j_r =1}^d A_{i_1j_1}\hdots A_{i_rj_r}T_{j_1\hdots j_r}$$
for all $(i_1,\hdots, i_r)\in [d]^r$. Note that $A\bullet T\in S^r(\RR^d)$, and so 
$A\in \RR^{d\times d}$ acts on the space of symmetric tensors $S^r(\RR^d).$

There exist a bijection between the space of symmetric tensors $S^r(\RR^d)$ {and} the homogenous polynomials of degree $r$ in variables $\xx=(x_1, \ldots, x_{d} )$:
\begin{equation}
    f_T(\xx) = \sum_{i_1,\ldots i_r = 1}^{d} T_{i_1\ldots i_r}x_{i_1}\cdots x_{i_r}.
\end{equation}
From this bijection, it is easy to see that $S^r(\RR^d)$ has dimension $\binom{d+r-1}{r}$.
We state the following lemma that will be used later:
\begin{lem}{\cite{nonica}}\label{lem:hessfAT}
 If $T\in S^r(\RR^d)$ and $A\in \RR^{d\times d}$, then $f_{A\bullet T}(x) = f_T(A^T x)$ and $\nabla^2 f_{A\bullet T} = A \nabla^2f_T(A^Tx)A^T.$
\end{lem}

Finally, we define the \emph{$i$-th marginalization} $T_{\ldots+\ldots}$ of a tensor $T\in S^r(\RR^d)$ as the sum of $T$ over the $i$-th coordinate. That is $$T_{j_1,\ldots,j_{i-1},+,j_{i+1},\ldots, j_r} = \sum_{{j_i}=1}^d T_{j_1,\ldots,j_{i-1},j_i,j_{i+1},\ldots, j_r}.$$

\subsection{Moment and cumulant tensors}

\par Let $\xx=(x_1,\hdots, x_d)$ be a vector of random variables. We write $\mu_r(\xx)\in S^r(\RR^d)$ to denote the \textit{$r$-th order moment tensor} of $\xx$ as the symmetric tensor with entries $$\mu_r(\xx)_{i_1,\hdots, i_r} = \EE [x_{i_1},\hdots, x_{i_r}].$$
Alternatively, the moments of $\xx$ can be defined using the \textit{moment generating function } $M_\xx(\ttt)=\EE[e^{\ttt^T \xx}].$ The logarithm of the moment generating function is the \textit{cumulant generating function}, denoted as $K_\xx(\ttt)= \log M_\xx(\ttt)=\log \EE [e^{\ttt^T\xx}]$. 
The entries of the moment $\mu_r(\xx)\in S^r(\RR^d)$ and cumulant $\kappa_r(\xx)\in S^r(\RR^d)$ tensors are expressed as follows: 
$$\mu_r(\xx)_{i_1,\hdots, i_r} =\frac{\partial^r}{\partial t_{i_1}\hdots \partial t_{i_r}} M_\xx(\ttt)\bigg|_{t=0},\quad \kappa_r(\xx)_{i_1,\hdots, i_r} = \frac{\partial^r}{\partial t_{i_1}\hdots \partial t_{i_r}} K_\xx(\ttt)\bigg|_{t=0}.$$

\par The moment and cumulant tensors of the random vector $\xx$ are related. In particular, $\kappa_1(\xx)=\mu_1(\xx)=\EE(\xx)$ and $\kappa_2(\xx)= \mu_2(\xx)-\mu_1(\xx)\mu_1(\xx)^T=\var (\xx).$ 
In what follows, we outline a well-established combinatorial relationship between moments and cumulants (see \cite{mom_cum}).
%
Let $\Pi_r$ be the poset of all partitions of $[r]$. We denote by $B$ the blocks of a partition $\pi \in \Pi_r$, and $i_B$ the subvector of $i$ with indices in $B.$
\begin{thm}\label{thm:combo}
    Consider a random vector $\xx\in \RR^d$. For any subset of indices $\{i_1,\hdots i_r\}$ of $\{1,\hdots d\}$ we can relate the moments $\mu_r(\xx)$ and cumulants $\kappa_r(\xx)$ as follows:
    \begin{align*}
        \mu_r(\xx)_{i_1 \hdots i_r}&=\sum_{\pi\in \Pi_r}\prod_{B\in \pi} \kappa_{|B|}(\xx)_{i_B} \text{ and}\\
        \kappa_r(\xx)_{i_1\hdots i_r}&=\sum_{\pi\in \Pi_r}(-1)^{|\pi|-1}(|\pi|-1)! \prod_{B\in \pi} \mu_{|B|}(\xx)_{i_B}.
    \end{align*}
\end{thm}
\begin{exmp}
    Let $r=3$, then $\Pi_r=\{123,1|23,2|13,3|12,1|2|3\}$. For simplicity we use the notation $\mu_{i_1 \hdots i_k}:= \mu_k(\xx)_{i_1\hdots i_k}$ and similarly for cumulants. Then,
    \begin{align*}
        \mu_{i_1i_2i_3}&= \kappa_{i_1i_2i_3}+\kappa_{i_1i_2}\kappa_{i_3}+\kappa_{i_1i_3}\kappa_{i_2}+\kappa_{i_2i_3}\kappa_{i_1}+\kappa_{i_1}\kappa_{i_2}\kappa_{i_3},\\
        \kappa_{i_1i_2i_3}&=\mu_{i_1i_2i_3}-\mu_{i_1i_2}\mu_{i_3}-\mu_{i_1i_3}\mu_{i_2}-\mu_{i_2i_3}\mu_{i_1}+2\mu_{i_1}\mu_{i_2}\mu_{i_3}.
    \end{align*}
\end{exmp}

\subsection{Properties of Cumulant Tensors}
Cumulants possess several properties that justify their use despite their added complexity compared to moments.

\par \textit{Multilinearity:} Similar to moments, cumulants exhibit multilinearity. For any matrix $A\in \RR^{d\times d}$, the following equation holds: $$h_r(A\xx)=A\bullet h_r(\xx).$$

\par \textit{Independence:} Given two mutually independent random vectors, $\xx$ and $\yy$, their cumulants exhibit an additive property $\kappa_r(\xx+\yy)=\kappa_r(\xx)+\kappa_r(\yy).$ The independence property is known as the cumulative property of cumulants. We can derive the cumulative property from the cumulant generating function as follows: 
\begin{align}\label{eq:independence}
    K_{\xx+\yy}(\ttt)&=\log \EE \left(e^{\langle \ttt, \xx+\yy\rangle}\right)=\log \EE \left(e^{\sum_i t_i(x_i+y_i)}\right)=\log \EE \left(e^{\sum_i t_ix_i+\sum_it_iy_i}\right)\\
    &=\log \EE \left(e^{\sum_i t_ix_i}e^{\sum_it_iy_i}\right)=\log \EE \left(e^{\sum_i t_ix_i}\right)+\log \EE\left(e^{\sum_it_iy_i}\right)=K_\xx(\ttt)+K_\yy(\ttt).\nonumber
\end{align}
Note that the second to last step follows from the independence of $\xx$ and $\yy$ and logarithmic properties. The independence property extends to multiple independent random vectors.

\par Additionally, independence among different components of $\xx$ results in zero entries in the cumulant tensors as demonstrated in Section \ref{sec:partIndep}.

\par \textit{Gaussianity:} The Gaussian distribution has a special property related to cumulants; higher-order cumulants of Gaussian distributions are all zero, as established by a version of the Marcinkiewicz classical result (see Lemma \ref{lem:Marc}).
\begin{prop}
    If $\xx \sim \mathcal{N}_d(\mu, \Sigma)$, then $\kappa_1(\xx)=\mu$, $\kappa_2(\xx)=\Sigma$, and $\kappa_r(\xx)=0$ for $r\ge 3.$ Moreover, the Gaussian distribution is the only probability distribution such that there exists an $r_0$ with the property that $\kappa_r(\xx)=0$ for all $r\ge r_0.$
\end{prop}

\subsection{Cumulant tensors and independence}\label{sec:cumIndep}

{In this section we aim to demonstrate how different independence conditions on the components on $\xx\in\RR^d$ translate into zero entries of the cumulant tensors $\kappa_r(\xx)$ of $\xx$. }

\subsubsection{Partitioned independence}\label{sec:partIndep} {Motivated by ISA, we consider the situation when a random vector 
consists of a partition of independent subvectors 
$\xx = (\xx_{1}, \ldots, \xx_{m})\in\RR^d$ where each subvector is of the form $\xx_{i} = (x_{i_k})_{i_k\in I_m}$ for some partition $I_1,\ldots, I_m$ of the indices $[d]$.
For simplicity, we state the following proposition for a partition with two elements $I$, $J$. However, it directly generalizes when considering $m>2$ partitions of $\xx$ into independent components.}


\begin{prop} \label{prop:partitioning}
    Let $I,J$ be a partition of $\{i_1,\hdots, i_r\}$. Then $x_i$ and $x_j$ are independent for any $i\in I$ and $j\in J$ if and only if $\kappa_r(\xx)_{i_1,\hdots, i_r}=0$.
\end{prop}

The proof of this proposition relies on the following theorem.

\begin{thm}[\cite{Jacod_Protter_2003}]\label{thm:uniquenessMGF}
    Assume $M_\xx(\ttt) = M_\yy(\ttt)$ for $t\in[-r,r]^n$ for some $r>0$, then the random vectors $\xx$ and $\yy$ have the same distribution.
\end{thm}

\begin{proof}[Proof of Proposition \ref{prop:partitioning}]
    For simplicity, we denote by $\xx_I$ (and similarly for $\xx_J$) the random subvector $(x_i)_{i\in I}.$ Consider the cumulant generating function, $K_\xx(\ttt)$. Applying similar steps as in Equation \eqref{eq:independence} we result in
        \begin{align*}
    K_{\xx}(\ttt)&=\log \EE \left(e^{\langle \ttt, \xx\rangle}\right)=\log \EE \left(e^{\sum_i t_ix_i}\right)=\log \EE \left(e^{\sum_{i\in I} t_ix_i+\sum_{j\in J}t_jx_j}\right)\\
    &=\log \EE \left(e^{\sum_{i\in I} t_ix_i}e^{\sum_{j\in J}t_jx_j}\right)=\log \EE \left(e^{\sum_{i\in I} t_ix_i}\right)+\log \EE\left(e^{\sum_{j\in J}t_jx_j}\right)=K_{\xx_I}(\ttt)+K_{\xx_J}(\ttt).
\end{align*}
The above holds since $\xx_I$ and $\xx_J$ are independent. We can also infer this from the cumulative property by considering two independent random vectors $(\xx_I, 0,\hdots, 0)\in \RR^d$ and $(0,\hdots, 0, \xx_J)\in \RR^d$.
From here, we may conclude that $$\kappa_r(\xx)_{i_1,\hdots, i_r}= \frac{\partial^r}{\partial t_{i_1}\hdots \partial t_{i_r}}K_\xx(\ttt)\bigg|_{t=0}= \frac{\partial^r}{\partial \ttt_I \partial \ttt_J}\left(K_{\xx_I}(\ttt_I)+K_{\xx_J}(\ttt_J)\right)\bigg|_{t=0}=0$$
where $\partial \ttt_I= \partial t_{j_1} \hdots \partial t_{j_s}$ for indices $j_i\in I$ (and similarly for $\partial \ttt_J$).

To prove the reverse, consider the Taylor expansion of the cumulant generating function $K_\xx(\ttt)$ at $\ttt=0$, that is,
\begin{align*}
    K_\xx(\ttt) = &K_\xx(0) + \sum_{i\in[d]}\kappa_1(\xx)_it_i + \sum_{i,j\in[d]}\frac{\kappa_2(\xx)_{ij}t_it_j}{2!} + \sum_{i,j,k\in[d]}\frac{\kappa_3(\xx)_{ijk}t_it_jt_k}{3!}+
    \cdots \\
    = &\sum_{i\in I}\kappa_1(\xx)_it_i + \sum_{j\in J}\kappa_1(\xx)_jt_j + \sum_{i_1,i_2\in I}\frac{\kappa_2(\xx)_{i_1i_2}t_{i_1}t_{i_2}}{2!} + \sum_{j_1,j_2\in J}\frac{\kappa_2(\xx)_{j_1j_2}t_{j_1}t_{j_2}}{2!} + \\
    &+ \sum_{i_1,i_2,i_3\in I}\frac{\kappa_3(\xx)_{i_1i_2i_3}t_{i_1}t_{i_2}t_{i_3}}{3!} + \sum_{j_1,j_2,j_3\in J}\frac{\kappa_3(\xx)_{j_1j_2j_3}t_{j_1}t_{j_2}t_{j_3}}{3!} + \cdots \\
    &=K_{\xx_I}(\ttt_I) + K_{\xx_J}(\ttt_J),
\end{align*}
where the second equality follows since all cross cumulants $\gencum=0$ for indices in $I$ and $J$ by definition.
This implies that $M_\xx(\ttt) = M_{\xx_I}(\ttt_I)M_{\xx_J}(\ttt_J)$, and therefore $M_\xx(\ttt)$ is the moment generating function of a vector whose components $\xx_I$ and $\xx_J$ are independent. By Theorem \ref{thm:uniquenessMGF} we can conclude that the random variables $\xx_I$ and $\xx_J$ are independent.
\end{proof}

{
\begin{cor}
    Let $\xx = (\xx_{1}, \ldots, \xx_{m}) \in\RR^d$ be a random vector with $\xx_{i} = (x_{i_k})_{i_k\in I_m}$ for some partition $I_1,\ldots, I_m$ of the indices $[d]$. The vectors $\xx_{j}$, $\xx_{k}$ are independent if and only if $\kappa_r(\xx)_{i_1,\hdots, i_r}=0$ for all set of indices $\{i_1,\hdots, i_r\}$ such that there exist $i_j,i_k\in \{i_1,\hdots, i_r\}$ with $i_j\in I_j$ and $i_k\in I_k$. \label{cor:independentpartition}
\end{cor}
}

{If partitions contain a single element, then we have the following proposition.}
\begin{cor}
    \label{prop:diagonal}
    The components of $\xx$ are mutually independent if and only if $\kappa_r(\xx)$ is a diagonal tensor for all $r\ge 2.$
\end{cor}
This result follows directly from Proposition \ref{prop:partitioning} and emphasizes the importance of investigating the independence assumption in ICA by looking at what happens when we include non-zero elements in relevant higher-order cumulant tensors.

\subsubsection{Mean Independence}

{
We also consider the case where variables 
are mean independent.} 
{Two variables $x_i$ and $x_j$ are \emph{mean independent} if $$\EE (x_i\mid x_j)=\EE (x_i).$$}
\begin{prop}[\cite{nonica}]\label{prop:meanIndepTensor}
    If $x_i$ is mean independent of $x_j$, then $\kappa_{ij\hdots jj}=0$ for all order $r\ge 2$ cumulant tensors.
    \begin{proof}
        We can assume $\EE x_i =0$, so then for all $\ell \ge 0$, $$\EE (x_i x_j^\ell) = \EE(x_j^\ell \EE(x_i\mid x_j))=\EE(x_i)\EE(x_j^\ell) =0.$$
        Now using the combinatorial relationship between cumulants and moments we conclude the desired result.
    \end{proof}
\end{prop}

{
This result can be extended to random vectors. 
We say that the components of $\xx\in\RR^d$ are mean independent if $x_i$ is mean independent of $x_j\in\xx_J$ for any $i\in [d]$ and any subset $J$ of $[d]\setminus i$. Immediately, we have the following generalization.

\begin{prop}
    If $\xx$ is mean independent 
    then $\kappa_{i j_2 \hdots j_r}=0$ for any $i\in [d]$ and any collection $j_2,\hdots , j_r$ of elements in $[d]\setminus i$. \label{prop:meanindsparsity}
\end{prop}}

\subsubsection{Independence induced by a graph}
In this section, we explore sparsity patterns on the cumulant tensors of a random vector $\xx$, if the independence on the entries of $\xx$ is given by the edges of a graph $G$. We will also connect this to the case where the vector $\xx$ has a partition of subsets of dependent random variables, but independent from each other.

The following definition illustrates the independence of $\xx$ given by a graph $G$.
\begin{defn}\label{def:indepGraph}
    Let $G=(V,E)$ be an undirected graph with vertices $V=\{1,\hdots, d\}.$ We will say that independence of a random vector $\xx\in \RR^d$ is induced by $G$ if each pair of variables $x_i$ and $x_j$ are independent if and only if there is no edge in $E$ connecting $i$ and $j.$
\end{defn}
The next result follows directly. 
\begin{prop}\label{prop:graph}
    Let $\xx \in \RR^d$ be a random vector with independence induced by a graph $G=(V,E)$. Then $\kappa_r(\xx)_{i_1,\hdots, i_r}=0$ if and only if the induced subgraph $G'=(V',E')$ formed by the subset of vertices $V'=\{i_1, \hdots, i_r\}$ and the corresponding subset of edges $E'=\{e :(u,v) \in E \mid u,v \in V'\}$ is a disconnected graph. 
    \begin{proof}

    {Since $\xx$ has independence induced by a graph $G$, a subvector $\xx'=(x_{i_1}, \ldots, x_{i_r})$ of $\xx$ will be partitioned into independent subvectors $\xx_I$ and $\xx_J$ if and only if the vertices in $I$ and $J$ are disconnected in $G'$.}
    Immediately the result follows from Corollary \ref{cor:independentpartition}. 
    \end{proof}
\end{prop}
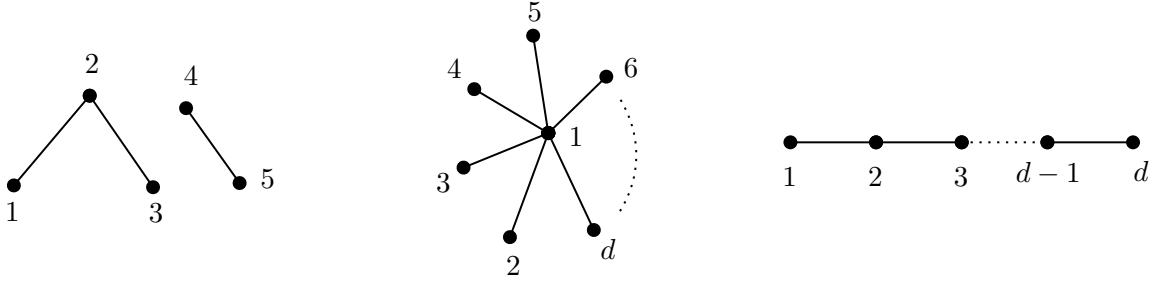
\begin{figure}[H]

    \tikzset{every picture/.style={line width=0.75pt}} 
    
    \begin{tikzpicture}[x=0.75pt,y=0.75pt,yscale=-0.9,xscale=0.9]
    
    \draw    (510.67,159.33) -- (558.67,159.33) ;
    \draw [shift={(558.67,159.33)}, rotate = 0] [color={rgb, 255:red, 0; green, 0; blue, 0 }  ][fill={rgb, 255:red, 0; green, 0; blue, 0 }  ][line width=0.75]      (0, 0) circle [x radius= 3.35, y radius= 3.35]   ;
    \draw [shift={(510.67,159.33)}, rotate = 0] [color={rgb, 255:red, 0; green, 0; blue, 0 }  ][fill={rgb, 255:red, 0; green, 0; blue, 0 }  ][line width=0.75]      (0, 0) circle [x radius= 3.35, y radius= 3.35]   ;
    \draw    (462.67,159.33) -- (510.67,159.33) ;
    \draw [shift={(510.67,159.33)}, rotate = 0] [color={rgb, 255:red, 0; green, 0; blue, 0 }  ][fill={rgb, 255:red, 0; green, 0; blue, 0 }  ][line width=0.75]      (0, 0) circle [x radius= 3.35, y radius= 3.35]   ;
    \draw [shift={(462.67,159.33)}, rotate = 0] [color={rgb, 255:red, 0; green, 0; blue, 0 }  ][fill={rgb, 255:red, 0; green, 0; blue, 0 }  ][line width=0.75]      (0, 0) circle [x radius= 3.35, y radius= 3.35]   ;
    \draw  [dash pattern={on 0.84pt off 2.51pt}]  (558.67,159.33) -- (606.67,159.33) ;
    \draw [shift={(606.67,159.33)}, rotate = 0] [color={rgb, 255:red, 0; green, 0; blue, 0 }  ][fill={rgb, 255:red, 0; green, 0; blue, 0 }  ][line width=0.75]      (0, 0) circle [x radius= 3.35, y radius= 3.35]   ;
    \draw [shift={(558.67,159.33)}, rotate = 0] [color={rgb, 255:red, 0; green, 0; blue, 0 }  ][fill={rgb, 255:red, 0; green, 0; blue, 0 }  ][line width=0.75]      (0, 0) circle [x radius= 3.35, y radius= 3.35]   ;
    \draw    (606.67,159.33) -- (654.67,159.33) ;
    \draw [shift={(654.67,159.33)}, rotate = 0] [color={rgb, 255:red, 0; green, 0; blue, 0 }  ][fill={rgb, 255:red, 0; green, 0; blue, 0 }  ][line width=0.75]      (0, 0) circle [x radius= 3.35, y radius= 3.35]   ;
    \draw [shift={(606.67,159.33)}, rotate = 0] [color={rgb, 255:red, 0; green, 0; blue, 0 }  ][fill={rgb, 255:red, 0; green, 0; blue, 0 }  ][line width=0.75]      (0, 0) circle [x radius= 3.35, y radius= 3.35]   ;
    \draw    (27.5,183.5) -- (70,133.17) ;
    \draw [shift={(70,133.17)}, rotate = 310.18] [color={rgb, 255:red, 0; green, 0; blue, 0 }  ][fill={rgb, 255:red, 0; green, 0; blue, 0 }  ][line width=0.75]      (0, 0) circle [x radius= 3.35, y radius= 3.35]   ;
    \draw [shift={(27.5,183.5)}, rotate = 310.18] [color={rgb, 255:red, 0; green, 0; blue, 0 }  ][fill={rgb, 255:red, 0; green, 0; blue, 0 }  ][line width=0.75]      (0, 0) circle [x radius= 3.35, y radius= 3.35]   ;
    \draw    (70,133.17) -- (105.5,184.5) ;
    \draw [shift={(105.5,184.5)}, rotate = 55.33] [color={rgb, 255:red, 0; green, 0; blue, 0 }  ][fill={rgb, 255:red, 0; green, 0; blue, 0 }  ][line width=0.75]      (0, 0) circle [x radius= 3.35, y radius= 3.35]   ;
    \draw [shift={(70,133.17)}, rotate = 55.33] [color={rgb, 255:red, 0; green, 0; blue, 0 }  ][fill={rgb, 255:red, 0; green, 0; blue, 0 }  ][line width=0.75]      (0, 0) circle [x radius= 3.35, y radius= 3.35]   ;
    \draw    (124,140.17) -- (154,182.17) ;
    \draw [shift={(154,182.17)}, rotate = 54.46] [color={rgb, 255:red, 0; green, 0; blue, 0 }  ][fill={rgb, 255:red, 0; green, 0; blue, 0 }  ][line width=0.75]      (0, 0) circle [x radius= 3.35, y radius= 3.35]   ;
    \draw [shift={(124,140.17)}, rotate = 54.46] [color={rgb, 255:red, 0; green, 0; blue, 0 }  ][fill={rgb, 255:red, 0; green, 0; blue, 0 }  ][line width=0.75]      (0, 0) circle [x radius= 3.35, y radius= 3.35]   ;
    \draw    (305.5,212.5) -- (327,154.17) ;
    \draw [shift={(327,154.17)}, rotate = 290.23] [color={rgb, 255:red, 0; green, 0; blue, 0 }  ][fill={rgb, 255:red, 0; green, 0; blue, 0 }  ][line width=0.75]      (0, 0) circle [x radius= 3.35, y radius= 3.35]   ;
    \draw [shift={(305.5,212.5)}, rotate = 290.23] [color={rgb, 255:red, 0; green, 0; blue, 0 }  ][fill={rgb, 255:red, 0; green, 0; blue, 0 }  ][line width=0.75]      (0, 0) circle [x radius= 3.35, y radius= 3.35]   ;
    \draw    (318.5,99.5) -- (327,154.17) ;
    \draw [shift={(327,154.17)}, rotate = 81.16] [color={rgb, 255:red, 0; green, 0; blue, 0 }  ][fill={rgb, 255:red, 0; green, 0; blue, 0 }  ][line width=0.75]      (0, 0) circle [x radius= 3.35, y radius= 3.35]   ;
    \draw [shift={(318.5,99.5)}, rotate = 81.16] [color={rgb, 255:red, 0; green, 0; blue, 0 }  ][fill={rgb, 255:red, 0; green, 0; blue, 0 }  ][line width=0.75]      (0, 0) circle [x radius= 3.35, y radius= 3.35]   ;
    \draw    (285.5,129.5) -- (327,154.17) ;
    \draw [shift={(327,154.17)}, rotate = 30.73] [color={rgb, 255:red, 0; green, 0; blue, 0 }  ][fill={rgb, 255:red, 0; green, 0; blue, 0 }  ][line width=0.75]      (0, 0) circle [x radius= 3.35, y radius= 3.35]   ;
    \draw [shift={(285.5,129.5)}, rotate = 30.73] [color={rgb, 255:red, 0; green, 0; blue, 0 }  ][fill={rgb, 255:red, 0; green, 0; blue, 0 }  ][line width=0.75]      (0, 0) circle [x radius= 3.35, y radius= 3.35]   ;
    \draw    (359.5,122.5) -- (327,154.17) ;
    \draw [shift={(327,154.17)}, rotate = 135.74] [color={rgb, 255:red, 0; green, 0; blue, 0 }  ][fill={rgb, 255:red, 0; green, 0; blue, 0 }  ][line width=0.75]      (0, 0) circle [x radius= 3.35, y radius= 3.35]   ;
    \draw [shift={(359.5,122.5)}, rotate = 135.74] [color={rgb, 255:red, 0; green, 0; blue, 0 }  ][fill={rgb, 255:red, 0; green, 0; blue, 0 }  ][line width=0.75]      (0, 0) circle [x radius= 3.35, y radius= 3.35]   ;
    \draw    (327,154.17) -- (279.5,173.5) ;
    \draw [shift={(279.5,173.5)}, rotate = 157.85] [color={rgb, 255:red, 0; green, 0; blue, 0 }  ][fill={rgb, 255:red, 0; green, 0; blue, 0 }  ][line width=0.75]      (0, 0) circle [x radius= 3.35, y radius= 3.35]   ;
    \draw [shift={(327,154.17)}, rotate = 157.85] [color={rgb, 255:red, 0; green, 0; blue, 0 }  ][fill={rgb, 255:red, 0; green, 0; blue, 0 }  ][line width=0.75]      (0, 0) circle [x radius= 3.35, y radius= 3.35]   ;
    \draw    (352.5,208.5) -- (327,154.17) ;
    \draw [shift={(327,154.17)}, rotate = 244.86] [color={rgb, 255:red, 0; green, 0; blue, 0 }  ][fill={rgb, 255:red, 0; green, 0; blue, 0 }  ][line width=0.75]      (0, 0) circle [x radius= 3.35, y radius= 3.35]   ;
    \draw [shift={(352.5,208.5)}, rotate = 244.86] [color={rgb, 255:red, 0; green, 0; blue, 0 }  ][fill={rgb, 255:red, 0; green, 0; blue, 0 }  ][line width=0.75]      (0, 0) circle [x radius= 3.35, y radius= 3.35]   ;
    \draw  [dash pattern={on 0.84pt off 2.51pt}]  (367.5,135.5) .. controls (380.5,157.5) and (378.5,181.5) .. (365.5,200.5) ;
    
    \draw (558.22,178.53) node  [font=\small]  {$3$};
    \draw (510.22,178.53) node  [font=\small]  {$2$};
    \draw (462.22,178.53) node  [font=\small]  {$1$};
    \draw (659.22,175.53) node  [font=\small]  {$d$};
    \draw (607.22,176.53) node  [font=\small]  {$d-1$};
    \draw (21,191.07) node [anchor=north west][inner sep=0.75pt]  [font=\small]  {$1$};
    \draw (121,115.07) node [anchor=north west][inner sep=0.75pt]  [font=\small]  {$4$};
    \draw (101.5,191.9) node [anchor=north west][inner sep=0.75pt]  [font=\small]  {$3$};
    \draw (65.67,108.4) node [anchor=north west][inner sep=0.75pt]  [font=\small]  {$2$};
    \draw (164,173.07) node [anchor=north west][inner sep=0.75pt]  [font=\small]  {$5$};
    \draw (336.75,149.23) node [anchor=north west][inner sep=0.75pt]  [font=\small]  {$1$};
    \draw (313.67,79.4) node [anchor=north west][inner sep=0.75pt]  [font=\small]  {$5$};
    \draw (268.67,111.4) node [anchor=north west][inner sep=0.75pt]  [font=\small]  {$4$};
    \draw (262.5,174.9) node [anchor=north west][inner sep=0.75pt]  [font=\small]  {$3$};
    \draw (301.67,221.4) node [anchor=north west][inner sep=0.75pt]  [font=\small]  {$2$};
    \draw (367.67,110.4) node [anchor=north west][inner sep=0.75pt]  [font=\small]  {$6$};
    \draw (354.5,211.9) node [anchor=north west][inner sep=0.75pt]  [font=\small]  {$d$};

    \end{tikzpicture}

    \caption{(L) Graph with 5 vertices and two connected components. \\ (M) Star tree $S_d$ with $d-1$ leaves. (R) Chain graph with $d$ vertices.}
    \label{fig:G1}
\end{figure}
\begin{exmp}
    Consider the graph $G$ in the left part of Figure \ref{fig:G1}. Consider a random vector $\xx=(x_1,x_2,x_3,x_4,x_5)$ with independence induced by $G.$ That is, $x_1 \indep x_3$ and $x_i \indep x_j$ for all $i=1,2,3$ and $j=4,5.$ Using the combinatorial formula in Theorem \ref{thm:combo} and the fact that moments factor due to independence, we can compute some entries of the 3rd and 4th order cumulants of $\xx$. 
    \par Consider the subgraph $G_0$ of $G$ formed by vertices $V_0=\{1,2,3\}$, it has a single connected component. Then, 
    {\small\begin{align*}
        \kappa_3(\xx)_{123} &= \mu_{123} - \mu_1 \mu_{23}-\mu_2 \mu_{13}-\mu_3 \mu_{12} +2\mu_1\mu_2\mu_3 
        =\mu_{123}-\mu_1\mu_{23}-\mu_1\mu_2\mu_3 -\mu_3 \mu_{12}+2\mu_1\mu_2\mu_3 \\ &=\mu_{123}-\mu_1\mu_{23}-\mu_3\mu_{12}+\mu_{1}\mu_2\mu_3.
\end{align*}}
On the other hand, the subgraphs $G_i$ formed by the following vertex sets $V_1 = \{1,2,4\}$, $V_2=\{1,3\}$, $V_3=\{1,2,4,5\}$ have two connected components. Then,
{\small\begin{align*}
    \kappa_3(\xx)_{124}=\ &\mu_{124}-\mu_1\mu_{24}-\mu_2\mu_{14}-\mu_4\mu_{12}+2\mu_1\mu_2\mu_4 = \mu_4\mu_{12} -\mu_1\mu_2\mu_4-\mu_1\mu_2\mu_4 +2\mu_1\mu_2\mu_4=0,\\
    \kappa_3(\xx)_{133} =\ &\mu_{133}-\mu_1\mu_{33}-\mu_3\mu_{13}-\mu_3\mu_{13}+2\mu_1\mu_3\mu_3 = \mu_1\mu_{33} -\mu_1\mu_{33}-2\mu_1\mu_3^2 +2\mu_1\mu_3^2=0,\\
    \kappa_4(\xx)_{1245}=\ &\mu_{1245}-\mu_{12}\mu_{45}-\mu_{14}\mu_{25}-\mu_{15}\mu_{24}-\mu_1 \mu_{245} -\mu_2 \mu_{145}-\mu_4 \mu_{125} - \mu_5 \mu_{124}+\\
&2\mu_1\mu_2\mu_{45}+2\mu_1\mu_4\mu_{25}+2\mu_2\mu_4\mu_{15}+2\mu_1\mu_5\mu_{24}+2\mu_2\mu_5\mu_{14}+2\mu_4\mu_5\mu_{12}-6\mu_1\mu_2\mu_4\mu_5\\
    =\ &\mu_{12}\mu_{45}-\mu_{12}\mu_{45}-\mu_{1}\mu_2\mu_4\mu_{5}-\mu_{1}\mu_2\mu_4\mu_{5}-\mu_1 \mu_2 \mu_{45} -\mu_1 \mu_2 \mu_{45}-\mu_4 \mu_5\mu_{12} - \mu_4\mu_5 \mu_{12}\\
+&2\mu_1\mu_2\mu_{45}+2\mu_1\mu_2\mu_4\mu_{5}+2\mu_1\mu_2\mu_4\mu_{5}+2\mu_1\mu_2\mu_4\mu_{5}+2\mu_1\mu_2\mu_{4}\mu_5+2\mu_{4}\mu_5\mu_{12}\\ -&6\mu_1\mu_2\mu_4\mu_5=0\\
\end{align*}}
    \end{exmp}

Proposition \ref{prop:graph} allows us to establish the sparsity pattern of cumulant tensors of random vectors with independence induced by some special classes of graphs.

\begin{cor}\label{cor:starTree}
    Let $S_{d}$ be a star graph with internal node $1$ and $d-1$ leaves $\{2,\ldots,d\}$ as in Figure \ref{fig:G1} (M). If $\xx \in \RR^d$ is a random vector with independence induced by the graph $S_d$ then $\kappa_r(\xx)_{i_1\ldots i_r} = 0$ if and only if the indices $i_1\ldots i_r$ are different from $1$ and not all of them are equal.
\end{cor}

\begin{cor}\label{cor:chainTree}
   Consider a chain graph $G$ as shown in Figure \ref{fig:G1} (R). If $\xx \in \RR^d$ has independence induced by $G$, then the entries $\kappa_r(\xx)_{i_1,\ldots,i_r}=0$ if {and only if} $|i_j-i_{j+1}|>1$ for some $j$ which corresponds to subgraphs with disconnected components.
\end{cor}

{A random vector $\xx = (\xx_{1}, \ldots, \xx_{m}) \in\RR^d$ with independent components $\xx_i$ can be understood as a random vector with independence induced by a graph with disconnected components. Therefore,} 
the following result is equivalent to {Corollary \ref{cor:independentpartition}} and follows from Proposition \ref{prop:graph}.
\begin{prop}
    \label{prop:subspacegraph}
    Consider a graph $G= I_{1}\sqcup \hdots \sqcup I_{m}$ where each disconnected component is complete. Let $\xx\in\RR^d$ be a random vector with independence induced by $G$. Then $\kappa_{i_1,\hdots, i_r}=0$ when at least two indices correspond to vertices in different components, i.e. $i_\ell\in I_{j}$ and $i_{\ell'}\in I_{{k}}$ for $\ell \neq \ell'$ and $j\neq k.$\\
\end{prop}

\section{Partitioned Independence Component Analysis}\label{sec:pica}

In this section, we consider a generalization of the classical ICA problem described in \eqref{eq:icastatement}. We will consider different independence assumptions on the random variables $\sss$ and study from an algebraic perspective to which extent the mixing matrix $A$ can be estimated.

\subsection{The PICA model}\label{sec:PICAmodel}
Let $\mathbf{I} = I_{1}\sqcup \hdots \sqcup I_{m}$ be a partition of $[d]$ such that $|I_i|=k_i$. Consider a random vector $\sss = (\sss_{1}, \ldots, \sss_{m})\in\RR^d$ where each subvector is of the form $\sss_{i} = (s_{i_1},\ldots,s_{i_{k}})\in\RR^{k_i}$ with $i_1, \ldots i_k\in I_i$ and is independent of the others. We will say that $\sss$ has partitioned independence given by $\mathbf{I}$.
The model is defined by the same equation as ICA,
$$\yy=A\sss$$
where $s_i$ and $s_j$ are independent if and only if $i$ and $j$ are in different subsets $I_i\neq I_j.$ That is, we assume that the vectors $\sss_{1}, \ldots, \sss_{m}$ are independent but dependencies within the entries of each vector $\sss_{i}$ are allowed. As in ICA, we assume that both the mixing matrix $A$ and the source signals $\sss$ are hidden and we observe the mixed signals $\yy$ that have the same independence as $\sss.$

{If the components inside each subvector $\sss_i$ are dependent, this is known as independence subspace analysis. However, we will refer to this model as \textit{Partitioned Independence Component Analysis} (PICA), since extra independent conditions can be assumed among the components of $\sss_i$.}

\subsection{Algebraic approach in Component Analysis}\label{sec:connectionAlgStats}

In this section, we want to review the problem of recovering the mixing matrix $A$, or alternatively the unmixing matrix $W$ in \eqref{eq:inverseICA}, from an algebraic perspective by using the cumulant tensors of $\sss$.
We aim to outline the connection between our understanding of component analysis in a statistical sense and an algebraic sense. We follow the approach in \cite{nonica}, and we characterize the recovery of $W$ from an algebraic perspective by using the zero restrictions on the cumulants of $\sss$ determined by the independence of $\sss$.

As shown in Section \ref{sec:ICA}, the goal of ICA is to recover the matrix $A$ from $\yy=A\sss$, or equivalently the matrix $W$ from $\sss=W\yy$ so that the original signals $\sss$ may be recovered.
Independence on the entries of $\sss$ can be translated into zero entries on the cumulant tensors $\kappa_{r}(\sss)$ of $\sss$ as we have shown in Propositions \ref{prop:partitioning}, \ref{prop:meanindsparsity}, and \ref{prop:graph}. In this section, instead of independence conditions in $\sss$, we will assume additional structure in the cumulant tensors $\kappa(\sss).$

More generally, consider an algebraic variety $\cV$ of symmetric tensors $S^r(\RR^d)$. 
%
By multilinearity of cumulants we have that $$W\bullet \kappa_r(\yy) =\kappa_r(W\yy)=\kappa_r(\sss),$$
which means that $W\bullet \kappa_r(\yy) \in \cV$ if and only if $\kappa_r(\sss)\in \cV.$

As seen in Section \ref{sec:ICA}, any possible candidate matrix for \eqref{eq:inverseICA} is of the form $B=QW$ such that $B\yy=\sss$ and $Q$ is orthogonal. Then we have the following, $$B\bullet \kappa_r(\yy) = (QW)\bullet \kappa_r(\yy)= Q\bullet (W\bullet \kappa_r(\yy))=Q\bullet \kappa_r(\sss).$$
Hence, $B\bullet \kappa_r(\yy)\in \cV$ if and only if $Q\bullet \kappa_r(\sss)\in \cV$.

Therefore, the identification of $W$ is reduced to the subset of orthogonal matrices that preserve the structure of $T=\kappa_r(\sss)$, or in other words, the subset {$\{Q\in O(d)\mid Q\bullet T\in \mathcal{V}\},$ {where $O(d)$ is the set of $d\times d$ orthogonal matrices.}} {For a generic tensor $T\in\cV$ we denote by}
$$\mathcal{G}_{T}(\mathcal{V}):= \{Q\in O(d)\mid Q\bullet T\in \mathcal{V}\}.$$
{the set of orthogonal matrices $Q$ such that $Q\bullet T\in \cV$.}
This construction can be summarized in the following proposition:
\begin{prop}[\cite{nonica}]\label{prop:connectionAlgStat}
    Consider the model \eqref{eq:inverseICA}. Suppose that for some $r\geq 3$ we have that $T=\kappa_r(\sss) \in \cV \subseteq S^r(\RR^d)$. Then $W$ can be identified up to the set $$\{QW : Q \in \mathcal{G}_T(\mathcal{V})\}.$$
\end{prop}

For identifiability purposes, one is interested in the cases when $\mathcal{G}_T(\mathcal{V})$ contains a single element, such that $W$ can be uniquely recovered, or is a finite structured subset. However, in this work, we aim to characterize the set $\mathcal{G}_T(\mathcal{V})$ for some particular algebraic sets $\cV.$

{That is, we consider the following general problem:
Let $T\in\cV$ be a generic tensor in some algebraic subset $\cV\subseteq S^r(\RR^n)$ of symmetric tensors. Then, we want to characterize the set $\mathcal{G}_T(\mathcal{V})$ of orthogonal matrices $Q$ such that $Q\bullet T\in\cV$.
}

This question was first introduced in \cite{nonica}, where the authors characterize $\mathcal{G}_T(\mathcal{V})$ for some varieties $\mathcal{V}$, and in particular when $\mathcal{V}$ is the set of diagonal tensors.
{Let $SP(d)$ denote the set of signed permutation matrices, then}, they prove the following result:

\begin{prop}[\cite{nonica}]\label{prop:SPdiagTensors}
    Let $T\in S^r(\RR^d)$ with $r\geq 3$ be a diagonal tensor with at most one zero entry on the diagonal. Then $Q \bullet T \in \cV$ if and only if $Q \in SP(d)$, i.e. $\mathcal{G}_T(\mathcal{V}){=} SP(d)$ 
\end{prop}

Together with Proposition \ref{prop:connectionAlgStat} the authors of \cite{nonica} can state the following result, which generalizes the classical result by Comon (see Theorem \ref{thm:comons}).

\begin{thm}[\cite{nonica}]\label{thm:diagT}
    Consider the model (\eqref{eq:icastatement}) with $\EE \yy=0$ and $\var (\yy)=I_d$, and suppose that for some $r\ge 3$ the tensor $\kappa_r(\yy)$ is diagonal with at most one zero on the diagonal. Then $A$ in \eqref{eq:icastatement} is identifiable up to permuting and swapping the signs of the rows (i.e. up to the action of $\operatorname{SP}(d)$).
\end{thm}

This result implies that one does not need to require independence of $\sss$; it is enough to have a diagonal cumulant tensor $\kappa_r(\sss)$ to assure that the matrix $A$ can be identified up to permuting and swapping the signs of the rows.
Our main goal is to study whether a similar result exists for the partitioned independence case.

{In the following sections we will consider algebraic sets $\cV$ of symmetric tensors motivated by independence on our source vector $\sss$, generalizing the classical component analysis.}

\subsection{Block diagonal cumulant tensors}\label{sec:algebraicPICA}

{
Let $\cV_\mathbf{I}$ be the variety of symmetric tensors $T\in S^r(\RR^d)$ with zero entries induced by the partition $\mathbf{I} = I_{1}\sqcup \hdots \sqcup I_{m}$ as follows 
\begin{align}\label{eq:Vpart}
\cV_\mathbf{I} =\{T\in S^r(\RR^d)\mid &T_{i_1\ldots i_r}=0 \text{ if } \exists i_{j}, i_{k}\in \{i_1,\ldots ,i_r\} \\
 &\text{ with } i_{j}\in I_{j}, i_{k}\in I_{k} \text{ and } I_{j}\neq I_{k} \}. \nonumber
\end{align}
If $T\in \cV_\mathbf{I}$, for any subset of indices $I_j\in\{I_1,\ldots,I_m\},$ we denote by $T_j$ the subtensor of $T$ with entries $T_{i_1,\ldots i_r}$ where $i_1,\ldots i_r\in I_j.$
}

{From now on, we will use the notation $\mathcal{B}_m(k_1,\ldots k_m)$ for block matrices, i.e., if $Q\in\mathcal{B}(k_1,\ldots k_m)$ then we interpret $Q$ as a matrix with $m^2$ blocks or submatrices $Q_{ij}$, $i,j\in\{1,\ldots,m\}$ of dimension $k_i\times k_j$. If all $k_i$ are equal to $k$ then we will use the notation $\mathcal{B}_m(k)$.}
Let $\mathcal{P}_m O(k)\subset \mathcal{B}_m(k)$ 
denote the set of block-orthogonal matrices, that is, each block of $Q\in\mathcal{P}_m O(k)$ is a $k\times k$ matrix that is either orthogonal or a zero matrix and such that there is exactly one non-zero block in each column and row of $Q$.  
For example, for $m=2$ $$\mathcal{P}_2 O(k)=\left\{\begin{pmatrix}
    Q_1&0\\
    0&Q_2
\end{pmatrix}\mid Q_1, Q_2\in O(k)\}\right\}\cup \left\{\begin{pmatrix}
    0&Q_1\\
    Q_2&0
\end{pmatrix}\mid Q_1,Q_2\in O(k)\}\right\}.$$
It is straightforward to see that matrices in $\mathcal{P}_m O(k)$ are orthogonal matrices.
{We denote by $\mathcal{P}_m O(k_1,\ldots,k_m)\subset \mathcal{B}_m(k_1,\ldots k_m)$ the set of orthogonal matrices such that 
there is exactly one non-zero block in each column and row of $Q$. Similarly, $\mathcal{P}_m SP(k_1,\ldots,k_m)\subset \mathcal{B}_m(k_1,\ldots k_m)$ is the set of block matrices with at most one non-zero block $Q_{ij}$ in each column and row and at most one non-zero entry in each row and column of the non-zero blocks $Q_{ij}$.}  

Let $T\in\cV_{\mathbf{I}}$ be a symmetric tensor with a partitioned structure as in \eqref{eq:Vpart}. We will start by showing that the set of block-orthogonal matrices belongs to $\mathcal{G}_T(\cV_\mathbf{I}).$

\begin{prop}\label{prop:inclusionBlockOrth}
    Let $\mathbf{I} = I_{1}\sqcup \hdots \sqcup I_{m}$ be a partition of $[d]$ where $|I_i|=k_i$. For any generic $T\in \cV_\mathbf{I}$ it follows that $\mathcal{P}_m O(k_1,\ldots,k_m)\subseteq\mathcal{G}_T(\cV_\mathbf{I}).$ 
    
    \begin{proof}
        Let $Q\in\mathcal{P}_m O(k)$ and take $i_1,\ldots,i_r$, such that there are two indices in different partitions. Without loss of generality assume $i_1\in I_1$ and $i_2 \in I_2$. Then,
        \begin{align*}
            (Q\bullet T)_{i_1\ldots i_r} &= \sum_{j_1,\ldots,j_r = 1}^d Q_{i_1j_1}\cdots Q_{i_rj_r} T_{j_1,\ldots,j_r} = \\
            &= \sum_{I\in\{I_1,\ldots,I_m\}}\left(\sum_{j_1,\ldots,j_r \in I} Q_{i_1j_1}\cdots Q_{i_rj_r} T_{j_1,\ldots,j_r} \right) =0
        \end{align*}
        where the last step follows from the fact that for any $I\in\{I_1,\ldots,I_m\}$, $Q_{i_1,j_1}Q_{i_2,j_2} = 0$ since $Q$ has at most one non-zero block in each column, and $Q_{i_1,j_1}$ and $Q_{i_2,j_2}$ belong to different blocks in the same column.
    \end{proof}
\end{prop}

However, {if $T$ is not generic, then the converse is not always true.} 
\begin{exmp}\label{ex:counterex}
    Consider a tensor $T\in S^3(\RR^4)$ and the partition $\mathbf{I} = \{1,2\}\sqcup\{3,4\}.$ Then, 
    $$T_{113}= T_{114}=T_{223} = T_{224} = T_{133}=T_{144}=T_{233}=T_{244}=0,$$ and assume
    $$T_{334}= 3T_{444},\quad 3T_{333} = T_{344},\quad  T_{112} = 3T_{222},\quad  3T_{111} = T_{122}.$$
    Then $Q\bullet T \in \cV_\mathbf{I}$ where 
    $$Q= \frac{1}{2}\begin{pmatrix}
        -1 & 1 &  1&   1 \\
        1 & -1 &  1 &  1\\
        1 &  1 & -1 &  1\\
        1 &  1 &  1 & -1
    \end{pmatrix}.$$ 
    That is, $Q\in O(4)$ preserves the sparsity pattern of $T$, however, $Q\not\in \mathcal{P}_2O(2).$ 
\end{exmp}

The following computational example allows us to characterize the matrices in $\mathcal{G}_T(\cV_\mathbf{I})$ where $T\in S^3(\RR^3)$ and $\mathbf{I} = \{1,2\}\sqcup\{3\}.$
\begin{exmp}
Let $T\in \cV_\mathbf{I} \subseteq S^3(\RR^3)$ and $\mathbf{I} = \{1,2\}\sqcup\{3\}.$ Using \texttt{Macaulay2} \cite{M2} we computed the $3\times 3$ orthogonal matrices $Q\in O(3)$ that preserve the sparsity pattern of $T$. It follows from \eqref{eq:Vpart} that the only non-zero entries in $T$ are $T_{111}$, $T_{112}$, $T_{122}$, $T_{222}$ and $T_{333}$. We consider the ideal $\texttt{I}$ generated by the zero entries of $(Q\bullet T)$, that is, $\texttt{I}=\langle (Q\bullet T)_{113},(Q\bullet T)_{123},(Q\bullet T)_{133},(Q\bullet T)_{223},(Q\bullet T)_{233} \rangle$. Let \texttt{O} be the ideal generated by the entries of $QQ^T-Id$. 
 Then, 
 \begin{align*}
 & \texttt{i1 : R = QQ[T\_\{111\}, T\_\{112\}, T\_\{122\}, T\_\{222\}, T\_\{333\}, q\_\{1,1\}..q\_\{3,3\}];}\\
 & \texttt{i2 : J = saturate(I, \{T\_\{111\}, T\_\{112\}, T\_\{122\}, T\_\{222\}, T\_\{333\}\});}\\
 & \texttt{i3 : J = eliminate(\{T\_\{111\}, T\_\{112\}, T\_\{122\}, T\_\{222\}, T\_\{333\}\}, J);}\\
 & \texttt{i4 : G = E + O;}\\
 & \texttt{i5 : B = intersect(O + ideal(q\_\{1,3\}, q\_\{2,3\}, q\_\{3,1\}, q\_\{3,2\}),}\\
 & \hspace{10.5em} \texttt{O + ideal(q\_\{3,3\}));}\\
 & \texttt{i6 : B == G}\\
 & \texttt{o6 =  true}\\
 \end{align*}

This computation implies that all possible matrices in $\mathcal{G}_T(\cV_\mathbf{I})$ are orthogonal matrices that belong to $\mathcal{Q}_1\cup \mathcal{Q}_2\subset \mathcal{B}_2(2,1)$ where
$$\mathcal{Q}_1 ={\small \left\{\begin{pmatrix}
   q_{11} & q_{12} & 0 \\
   q_{21} & q_{22} & 0 \\
   0 & 0 & \pm 1 \\
\end{pmatrix}\in \mathcal{P}_2O(2,1) \right\} \mbox{ and }
\mathcal{Q}_2 = \left\{\begin{pmatrix}
   q_{11} & q_{12} & q_{13} \\
   q_{21} & q_{22} & q_{23} \\
   q_{31} & q_{32} & 0 \\
\end{pmatrix}\in O(3) \middle\vert  q_{11}q_{22} - q_{12}q_{21}=0 
 \right\}.}$$

First, note that $ \mathcal{Q}_1 = \mathcal{P}_2O(2,1)$. Moreover, any matrix $Q_1\in \mathcal{Q}_1$ preserves the sparsity patterns of tensors $T \in \cV_\mathbf{I}$ as shown in Proposition \ref{prop:inclusionBlockOrth}. However, simulations have shown us that while there always exist matrices $Q_2\in \mathcal{Q}_2$ such that $Q_2\bullet T\in \cV_{\mathbf{I}}$ for any generic tensor $T\in\cV_\mathbf{I}$, the entries of such matrices depend on the entries of $T$. That is, for any generic tensor $T\in \cV_\mathbf{I}$ and any matrices $Q_1\in \mathcal{Q}_1$ and $Q_2\in \mathcal{Q}_2$, it follows that $Q_1\bullet T\in \cV_{\mathbf{I}}$ but $Q_2\bullet T$ is not necessarily in $\cV_{\mathbf{I}}$.
\end{exmp}

These examples together with some performed computations motivated us to state the following result {which, under some mild assumptions allows us to characterize $\cV_{\mathbb{I}}$}.

\begin{thm}\label{thm:algPICA}
    Let $T \in S_r(\RR^d)$ be a generic tensor with $r\geq 3$ and consider $T\in\cV_\mathbf{I}$ where $\mathbf{I} = I_{1}\sqcup \hdots \sqcup I_{m}$ 
    with $|I_i|=k_i$. Let $Q\in \mathcal{B}_m(k_1,\ldots, k_m)$ be an orthogonal matrix with blocks $Q_{ij}$, $1\leq i,j, \leq m$. 
    If each block $Q_{ij}$ is either full rank or a zero matrix, then $Q\bullet T \in \cV_\mathbf{I}$ if and only if $Q = \mathcal{P}_m O(k_1,\ldots,k_m)$.
\end{thm}

We follow a similar strategy as the one used in Proposition \ref{prop:SPdiagTensors} to prove this result.
\begin{proof}
    The {reverse} implication is clear from Proposition \ref{prop:inclusionBlockOrth}. For the {forward} one, suppose $Q\in \mathcal{G}_T(\cV_\mathbf{I})$, that is, $Q\bullet T \in \cV_\mathbf{I}$. We want to prove that $Q\in \mathcal{P}_m O(k_1,\ldots,k_m).$

    Let $Q\in \mathcal{B}_m(k_1,\hdots, k_m)$, that is, it has the form
    $$ Q = 
    \bordermatrix{ & \textcolor{gray}{k_1} & \textcolor{gray}{k_2} & & \textcolor{gray}{k_m}\cr
     \textcolor{gray}{k_2} & Q_{11} & Q_{12} & \cdots & Q_{1m} \cr
     \textcolor{gray}{k_3} & Q_{21} & Q_{22} & \cdots & Q_{2m} \cr
     \textcolor{gray}{}    & \vdots &  &  &  \cr
     \textcolor{gray}{k_m} & Q_{m1} & Q_{m2} & \cdots & Q_{mm} \cr}. 
    $$ 
    where each block $Q_{ii}$ is a $k_i\times k_i$ full rank or zero matrix.

    Recall that for any subset of indices $I_j\in\{I_1,\ldots,I_m\},$  $T_j$ is the subtensor of $T$ with entries $T_{i_1,\ldots i_r}$ where $i_1,\ldots i_r\in I_j.$ Similarly, we use the notation $\xx_j$ for the subvector of $\xx\in\RR^d$ with entries $\left\{x_i\right\}_{i\in I_j}$.

    From Lemma \ref{lem:hessfAT} we have that $\nabla^2 f_{Q\bullet T}(\xx) = Q \nabla^2f_T(Q^T\xx)Q^T,$ or equivalently 
    \begin{equation}\label{eq:nabla}
        \nabla^2 f_{Q\bullet T}(\xx) Q = Q  \nabla^2f_T(Q^T\xx).
    \end{equation}
    Since $T\in\cV_\mathbf{I}$, the polynomial $f_{T}(\xx)$ can be decomposed as the sum of polynomials $f_j$ with variables $\xx_j$, that is,
    $$f_{T}(\xx) = \sum_{j=1}^m \left(\sum_{i_1,\ldots,i_r\in I_j} T_{i_1,\ldots,i_r}\xx_{i_1},\cdots \xx_{i_r} \right) = \sum_{j=1}^m f_{T_j}(\xx_{j}),$$ which implies that $\nabla^2 f_{T}$ is a block diagonal matrix. Thus, $\nabla^2 f_{Q\bullet T}$ is also a block diagonal matrix of the form
    $$\nabla^2 f_{Q\bullet T} = \begin{pmatrix}
        A_1(\xx_1) & 0 & \cdots & 0 \\
        0 & A_2(\xx_2)  & \cdots & 0 \\
        \vdots &  &  &  \\
        0 & 0 & \cdots & A_m(\xx_{m}) \\
    \end{pmatrix}$$ where each $A_i(\xx_i)$ is a $k_i\times k_i$ matrix with entries polynomials in the variables $\xx_{i}$. Similarly, $\nabla^2f_T(Q^T\xx)$ is also a block diagonal matrix with diagonal blocks $B_i(\xx)$, with entries polynomials in $\xx$.

    Therefore, Equation \eqref{eq:nabla} is equivalent to 
\begin{align*}
\scalemath{0.95}
    {\begin{pmatrix}
        A_1(\xx_{1}) & 0 & \cdots & 0 \\
        0 & A_2(\xx_{2})  & \cdots & 0 \\
        \vdots &  &  &  \\
        0 & 0 & \cdots & A_m(\xx_{m}) \\
    \end{pmatrix} }
    Q &= Q
    \scalemath{0.95}{\begin{pmatrix}
        B_1(\xx) & 0 & \cdots & 0 \\
        0 & B_2(\xx)  & \cdots & 0 \\
        \vdots & \vdots &  & \vdots \\
        0 & 0 & \cdots & B_m(\xx) \\
    \end{pmatrix}} \\
     \scalemath{0.95}{\begin{pmatrix}
        A_1(\xx_{1})Q_{11} & A_1(\xx_{1})Q_{12} & \cdots & A_1(\xx_{1})Q_{1m} \\
        A_2(\xx_{2})Q_{21} & A_2(\xx_{2})Q_{22} & \cdots & A_2(\xx_{2})Q_{2m} \\
        \vdots &  &  &  \\
        A_m(\xx_{m})Q_{m1} & A_m(\xx_{m})Q_{m2} & \cdots & A_m(\xx_{m})Q_{mm} \\
    \end{pmatrix}} &= 
    \scalemath{0.95}{\begin{pmatrix}
        Q_{11}B_1(\xx) & Q_{12}B_2(\xx) & \cdots & Q_{1m}B_m(\xx) \\
        Q_{21}B_1(\xx) & Q_{22}B_2(\xx) & \cdots & Q_{2m}B_m(\xx) \\
        \vdots & \vdots &  & \vdots \\
        Q_{m1}B_1(\xx) & Q_{m2}B_2(\xx) & \cdots & Q_{mm}B_m(\xx) \\
    \end{pmatrix},}
\end{align*} 

which implies that $A_i(\xx_{i})Q_{ij} = Q_{ij}B_j(\xx)$.
Assume there is a column block $j$ of $Q$ such that there are two submatrices $Q_{ij}$ and $Q_{kj}$ different from zero; otherwise, we are done. 

First, suppose that $i,k\geq j$, that is, the submatrices $Q_{ij}$ and $Q_{kj}$ have left inverses. Therefore
    \begin{equation*}
    \begin{cases}
       Q_{ij}^{-1}A_i(\xx_{i})Q_{ij} = B_j(\xx)\\
       Q_{kj}^{-1}A_k(\xx_{k})Q_{kj} = B_j(\xx)\\
    \end{cases}
\end{equation*} which implies that \begin{equation}\label{eq:AiAk}
    Q_{ij}^{-1}A_i(\xx_{i})Q_{kj} = Q_{kj}^{-1}A_k(\xx_{k})Q_{kj}.
\end{equation} Since, the entries of $A_i(\xx_{i})$ and $A_k(\xx_{k})$ are polynomials in different variables and $Q_{ij}$ and $Q_{kj}$ have null left kernel, \eqref{eq:AiAk} only holds if $A_i(\xx_{i}) = 0 = A_k(\xx_{k})$. The entries of any $A_i(\xx_{i})$ are non-zero polynomials for generic tensors $T$, so we get a contradiction. Therefore, $Q$ must be a block orthogonal matrix.

For the second case, suppose without loss of generality, that $i<j$. In this case, $Q_{ij}$ has a right inverse. This implies that $A_i(\xx_{i}) = Q_{ij}B_j(\xx)Q_{ij}^{-1}$. Note that the $(u,v)-$entry of $B_j(\xx)$ corresponds to $\frac{\partial^2}{\partial (\xx_{j})_u \partial (\xx_{j})_v} f_T(Q^T\xx),$ which means that the entries of $B_j(\xx)$ are polynomials in variables $\left(Q^T\xx\right)_{i\in I_j}$, or in other words $\sum_{l=1}^{m} (Q_{lj})^T\xx_l$. In particular, since we assume there are two non-zero blocks $Q_{ij}$ and $Q_{kj}$ in the same column, the entries of $B_j(\xx)$ are polynomials in $\xx_{i}$ and $\xx_{k}$. 
Decompose $B_j(\xx)$ as a sum of matrices $\tilde{B}^{(i)}(\xx_{i}) + \tilde{B}(\xx)$, such that the entries of $\tilde{B}^{(i)}(\xx_{i})$ are polynomials with variables $\xx_{i}$ and $\tilde{B}(\xx)$ polynomials on variables $\xx_{i}$ and $\xx_{k}$. Therefore,  
$$A_i(\xx_{i})Q_{ij} = Q_{ij}B_j(\xx) = Q_{ij}(\tilde{B}^{(i)}(\xx_{i}) + \tilde{B}(\xx)) = Q_{ij}\tilde{B}^{(i)}(\xx_{i})  + Q_{ij}\tilde{B}(\xx)$$
which implies that $Q_{ij}\tilde{B}(\xx)$ must be zero. 
Since $Q_{ij}$ has full rank and $i<j$, $Q_{ij}$ does not have a right kernel, and therefore we reach contradiction since $\tilde{B}(\xx)$ is non-zero generically for generic tensors $T$ and $r\geq 3$.
\end{proof}

Proposition \ref{prop:connectionAlgStat} allows us to state the following theorem.
\begin{thm}\label{thm:generalThmAlgPICA}
    Consider the model $\yy=A\sss$ with $\EE \yy=0$ and $\var (\yy)=I_d$, and suppose that for some $r\ge 3$ the tensor $\kappa_r(\yy)\in\cV_{\mathbf{I}}$ where $\mathbf{I} = I_{1}\sqcup \hdots \sqcup I_{m}$ and $|I_i|=k_i$ with no null blocks  $(\kappa_r(\yy))_j$. If $A\in\mathcal{B}_m(k_1,\ldots,k_m)$ has a block structure 
    and each block $A_{ij}$ is either zero or full rank, then $A$ is identifiable {up to the action of $\mathcal{P}_m O(k_1,\ldots, k_m)$.}
\end{thm}

In this case, $A$ is not identifiable up to a finite set. However, this result allows us to deal with any cumulant tensors of random vectors $\sss$ with a partition into independent subvectors $(\sss_1,\ldots,\sss_m)$. In the following sections, we will study cases where certain independence exist within each subvector $\sss_i$.

\subsection{Reflectionally invariant tensors}
{
A tensor $T\in S^r(\RR^d)$ is \emph{reflectionally invariant} if the only potentially non-zero entries of $T$ are the entries $T_{i_1,\ldots,i_r}$ such that each index $i_j$ appears an even number of times.
\begin{thm}[Theorem 5.10, \cite{nonica}]\label{thm:reflectioninvar}
    Suppose that $T\in S^r(\RR^d)$ for an even $r$ is a reflectionally invariant tensor satisfying  
    \begin{equation}\label{eq:condRI}
        T_{+\cdots + ii} \neq T_{+\cdots + jj} \mbox{ for all } i \neq j.
    \end{equation} Then $Q\bullet T$ is also reflectionally invariant if and only if $Q\in SP(d).$
\end{thm}
}

{From here, we can state the following corollary.
\begin{cor}
    Let $T \in S^r(\RR^d)$ be a generic tensor with even $r> 3$ such that $T\in\cV_\mathbf{I}$ where $\mathbf{I} = I_{1}\sqcup \hdots \sqcup I_{m}$ with $|I_i|=k$. Assume that, for any $i\in\{1,\ldots,m\}$, $T_i$ is a reflectionally invariant tensor that satisfies \eqref{eq:condRI}. Let $Q\in \mathcal{B}_m(k)$ be an orthogonal matrix with $k\times k$ blocks $Q_{ij}$, $1\leq i,j \leq m$.
    If each block $Q_{ij}$ is either full rank or a zero matrix, then $Q\bullet T \in \cV_\mathbf{I}$ if and only if $Q \in \mathcal{P}_m SP(k)$.
    \begin{proof}
        From Theorem \ref{thm:algPICA}, we know that $Q\in \mathcal{P}_m O(k)$. 
        Since $Q\bullet T\in\cV_{\mathbf{I}}$ its only non-zero entries are $(Q\bullet T)_{i_1,\ldots, i_r}$ where $i_1,\ldots, i_r\in I_i$ for some $i$. For any $I_i$ and sequence of indices $(i_1,\hdots, i_r)\in I_i$ we have $$(Q\bullet T)_{i_1,\hdots, i_r}=\sum_{j=1}^m\sum_{j_1,\hdots, j_r \in I_j} Q_{i_1j_1}\hdots Q_{i_r j_r}T_{j_1\hdots j_r} = \sum_{j_1,\hdots, j_r \in I_j} Q_{i_1j_1}\hdots Q_{i_r j_r}T_{j_1\hdots j_r} \text{ for some } I_j$$
        where the second equality follows since there is only one non-zero block $Q_{ij}$ in the $i$-th block row of $Q$.
        Therefore, there is only one submatrix $Q_{ij}$ acting on $T_{j}$ (since all other submatrices in the $i$-th row and $j$-th column of $Q$ are zero matrices). This implies that we can just consider $Q_{ij}\bullet T_{j}$, and since $T_{j}$ is reflectionally invariant, using Theorem \ref{thm:reflectioninvar} we conclude that $Q_{ij}\in SP(k)$, and thus $Q \in \mathcal{P}_m SP(k)$ as desired.
    \end{proof}
\end{cor}
Together with proposition \ref{prop:connectionAlgStat} we have the following result
\begin{cor}\label{cor:reflInvTensor}
    Consider the model $\yy=A\sss$ with $\EE \yy=0$ and $\var (\yy)=I_d$, and suppose that for some even $r>3$ the tensor $\kappa_r(\yy)\in\cV_{\mathbf{I}}$ where $\mathbf{I} = I_{1}\sqcup \hdots \sqcup I_{m}$ such that for each $i\in 1,\ldots m$, $(\kappa_r(\yy))_j$ is a non-zero reflectionally invariant tensor that satisfies \eqref{eq:condRI}. If $A\in\mathcal{B}_m(k_1,\ldots,k_m)$ has a block structure 
    and each block $A_{ij}$ is either zero or full rank, then $A$ is identifiable {up to the action of $\mathcal{P}_m SP(k_1,\ldots, k_m)$.}
\end{cor}

}

{
\subsection{Tensors with zero structure given by graphs}
In this section, we consider tensors whose zero entries are induced by a graph $G$. That is, we will first define the variety $\cV_G$ of symmetric tensors with zero entries induced by a graph $G$, then we will present two conjectures regarding the characterization of $\mathcal{G}_T(\cV_G)$ and give some examples for special classes of trees. 

Consider a graph $G=(V,E)$ with vertices $V=\{1,\ldots,d\}$. Motivated by Proposition \ref{prop:graph}, we define $\cV_G$ as follows
\begin{align}\label{eq:V_G}
\cV_G =\{T\in S^r(\RR^d)\mid &T_{i_1\ldots i_r}=0 \text{ if the subgraph $G'=(V',E')$ induced by the} \\ &\text{subset of vertices }V'=\{i_1,\ldots,i_r\} \text{ is disconnected}\}. \nonumber
\end{align}

Denote by $\mathcal{P}(d)$ the group of $d\times d$ permutation matrices. Then, the following conjecture is motivated by the fact that an automorphism of $G$ is a permutation matrix $P\in \mathcal{P}(d)$ that satisfies  $P^TAP=A$, where $A$ is the adjacency matrix of $G$. 

\begin{conj}
    Let $A$ be the adjacency matrix of a graph $G$.  If $Q\in SP(d)$, and $P\in\mathcal{P}(d)$ is a permutation matrix such that $P_{ij} = |Q_{ij}|$, then $Q\in\mathcal{G}_T(\cV_G)$ if and only if $P^TAP=A$.
\end{conj}

We will illustrate this conjecture with a couple of examples.

\begin{exmp}
    Consider a star tree $S_d$ as in Figure \ref{fig:G1} (M) with adjacency matrix $$A = \begin{pmatrix}
        0 & 1 & \ldots & 1 \\
        1 & 0 & \ldots & 0 \\
        \vdots &\vdots & & \vdots\\
        1 & 0 & \ldots & 0
    \end{pmatrix}.$$ Let $P = (p_{ij})_{i,j\in[d]}\in \mathcal{P}(d)$ be a permutation matrix. 
    We will show that $P$ defines an automorphism on $S_d$ if and only if $P = \begin{pmatrix}
        1 & \mathbf{0}^T \\
        \mathbf{0} & P'
    \end{pmatrix}$ where $P'\in \mathcal{P}(d-1)$ and $\mathbf{0}^T=(0,\ldots,0)^T\in\RR^{d-1}$. 
    That is, $P$ defines an automorphism on $S_d$ if and only $P^TAP=A$, or equivalently,
    $$AP = \begin{pmatrix}
        \sum_{i=2}^d p_{i1}  & \ldots & \sum_{i=2}^d p_{im}\\
        p_{11} & \ldots & p_{15} \\
        \vdots \\
        p_{11} & \ldots & p_{m1}\\
    \end{pmatrix} =
    \begin{pmatrix}
        \sum_{i=2}^d p_{1i} & p_{11} & \ldots & p_{11}\\
        \vdots \\
        \sum_{i=2}^d p_{mi} & p_{m1} & \ldots & p_{m1}\\
    \end{pmatrix}
     = PA.
    $$
    It follows that $p_{11}$ must be equal to $1$. Otherwise, if $p_{11}=0$ it is satisfied that $\sum_{i=2}^d p_{ij}=0$ for all $j\geq 2$ and therefore $p_{1j}= 1 \ \forall j\geq 2$ which contradicts with the fact that $P$ is a permutation matrix. If $p_{11}= 1$ it directly follows that $Q = \begin{pmatrix}
       1 & \mathbf{0}^T \\
        \mathbf{0} & P'
    \end{pmatrix}$ and $P'\in \mathcal{P}(d-1)$.

    Let $T\in \mathcal{V}_{S_d}$, we will see that if $Q = \begin{pmatrix}
        \pm 1 & \mathbf{0}^T \\
        \mathbf{0} & Q'
    \end{pmatrix}$, where $Q'\in SP(d-1)$, then $Q\in\mathcal{G}_T(\cV_{S_d})$. Similar to Corollary \ref{cor:starTree}, from \eqref{eq:V_G} we deduce that $T_{i_1,\ldots, i_r}=0$ if $1\not\in \{i_1,\ldots, i_r\}$ and not all of the indices ${i_1,\ldots, i_r}$ are equal. 
    That is take a set of indices $\{i_1,\ldots, i_r\}$ satisfying the previous two conditions, then
    $$ (Q\bullet T)_{i_1,\ldots, i_r} = \sum_{j_2,\ldots,j_r} Q_{i_11}Q_{i_2j_2}\cdots Q_{i_rj_r} T_{1j_2\ldots j_r} + \sum_{j\neq 1} Q_{i_1j}\cdots Q_{i_rj} T_{j\ldots j}$$
which equals zero since $Q_{i_11}=0$ for $i_1>2$ and $Q_{i_1j}\cdots Q_{i_rj}=0$ for any $j\in[d]$ since it is a product of entries of the same column of $Q$.
\end{exmp}

\begin{exmp}
 In this example we consider a chain tree $G$ as shown in Figure \ref{fig:G1} (R). Note that the adjacency matrix of $G$ is $$A = \begin{pmatrix}
     0 & 1 & 0 & 0 & \ldots & 0 & 0 \\
     1 & 0 & 1 & 0 & \ldots & 0 & 0 \\
     0 & 1 & 0 & 1 & \ldots & 0 & 0 \\
     \vdots \\
     0 & 0 & 0 & 0 & \ldots & 0 & 1 \\
     0 & 0 & 0 & 0 & \ldots & 1 & 0
 \end{pmatrix}.$$
 If a permutation matrix $P = (p_{ij})_{i,j\in[d]}\in \mathcal{P}(n)$ satisfies $AP=PA$, then it follows that the sum of the first row of $AP$ and $PA$ must be equal. This means, $\sum_i (AP)_{1i} = \sum_i p_{2i}$ equals $ \sum_i (PA)_{1i}=p_{11} + p_{1m} + 2\sum_{i=2}^{m-1}p_{1i}$, and therefore $ p_{11} + p_{1m} + 2\sum_{i=2}^{m-1}p_{1i}= 1$. Since $P$ is a permutation matrix, this implies that $p_{1i}$ must be zero for all $2\leq i\leq m-1$ and then, either $p_{11}=1$ or $p_{1m}=1$. Summing up the last rows of $AP$ and $PA$ we can also deduce that in the last row of $P$ either $p_{m1}=1$ or $p_{mm}=1$.

 If $p_{11} = p_{mm} = 1$, we have that $(AP)_{21} = 1+p_{31} = p_{22} = (PA)_{21}$ and therefore $p_{22}=1$. With a recursive argument, we deduce that $p_{ii} =1$ for all $i\in[d]$. A similar argument can be used if $p_{1m} = p_{m1} = 1$, deducting that $p_{i,m+i-1}=1.$
 Therefore $P$ is an automorphism of $G$ if and only if the non-zero entries of $P$ are $p_{ii}$ or $p_{i,m-i+1}$, that is $P$ is either a diagonal or an anti-diagonal matrix with non-zero entries equal to $1$.

 If $T\in\cV_{G}$, from \eqref{eq:V_G} we have that $T_{i_1\ldots i_r} =0$ if $|i_j-i_{j+1}|>1$ for some $j$ (similar to Corollary \ref{cor:chainTree}). Let $Q\in SP(D)$ be a diagonal or anti-diagonal matrix with non-zero entries equal to $\pm 1$. For any set of indices ${i_1\ldots i_r}$ such that $|i_j-i_{j+1}|>1$ for some $j$ we have that
 $$(Q\bullet T)_{i_1\ldots i_r} = \sum_{j_1,\ldots,j_r} Q_{i_1j_1}Q_{i_2j_2}\cdots Q_{i_rj_r} T_{j_1j_2\ldots j_r}=0$$
 since the only non-zero terms $T_{j_1j_2\ldots j_r}$ satisfy $|i_j-i_{j+1}|\leq 1$ for all $j$ and $Q_{i_1j_1}Q_{i_2j_2}\cdots Q_{i_rj_r} \neq 0$ if and only if $|i_k-i_{k+1}|\leq 1$ and $|j_k-j_{k+1}|\leq 1$ for all $k$. Therefore $Q\in\mathcal{G}_T(\cV_G)$.
\end{exmp}

We conclude the section with a stronger conjecture for trees.

\begin{conj}
Let $G$ be a forest, such that each connected component is a tree with at least $3$ nodes. 
    Let $A$ be the adjacency matrix of $G$ and $Q\in O(d)$ an orthogonal matrix. Then, $Q\in\mathcal{G}_T(\cV_G)$ if and only if $Q\in SP(d)$ and $P^TAP=A$, where $P\in\mathcal{P}(d)$ is a permutation matrix such that $P_{ij} = |Q_{ij}|$.
\end{conj}

}

\section{Identifiability of Partitioned Independence Component Analysis}\label{sec:statisticalPICA}

In this section, we want to review the proof of Comon's Theorem \ref{thm:comons} for the case of partitioned independence. This relies on the generalization of the Skitovich-Darmois theorem for the case of random vectors given by Ghurye and Olkin \cite{Ghurye1962}.
We start by generalizing Lemma \ref{lem:darLem}, Theorem \ref{thm:darm} and Corollary \ref{cor:darmois} to the case of independent random vectors, instead of independent random variables and we will see that it is not sufficient to extend Comon's theorem to independent vectors. The following two results can be found in \cite{bookVectors}.

\begin{lem} \label{lem:finitediffgen}
    Consider $$\sum_{i=1}^m f_i (\ttt_1+C_i \ttt_2)=g(\ttt_1)+h(\ttt_2),$$
    where $\ttt_1,\ttt_2 \in \RR^n$. If $C_i$ are nonsingular matrices with 
    $\det(C_i-C_j)\neq 0$ 
    for $i\neq j,$ 
    then the functions $f_i,$ $g$, and $h$ are polynomials of degree at most $m.$
    \begin{proof}
        Let $\alpha_i^{(1)}=(I-C_iC_m^{-1}) \alpha^{(1)}$. Then 
        \begin{align}
            \sum_{i=1}^m f_i(\ttt_1+\alpha_i^{(1)}+C_i\ttt_2)&= \sum_{i=1}^m f_i\left(\ttt_1+\alpha^{(1)}+C_i(\ttt_2-C_m^{-1}\alpha^{(1)})\right) \nonumber \\
            &=g(\ttt_1+\alpha^{(1)})+h(\ttt_2-C_m^{-1}\alpha^{(1)}). \label{eq:vec}
        \end{align}
        Subtracting Equation \eqref{eq:vec} from the initial, we have $$\sum_{i=1}^{m-1} f_i^{(1)}(\ttt_1+C_i\ttt_s) =g^{(1)}(\ttt_1)+h^{(1)}(\ttt_2),$$
        where $f^{(1)}_i(\xx)= f_i\left(\xx+\alpha^{(1)}_i\right)-f_i(\xx)$ (i.e. finite differences from before). With a similar construction as the one used in Lemma \ref{lem:finiteness}, we arrive at $$f_1^{(m-1)}(\ttt_1+C_1\ttt_2) = g^{(m-1)}(\ttt_1)+h^{(m-1)}(\ttt_2),$$
        where $f_1^{(m-1)}(\xx)= f_1^{(m-2)}\left(\xx+\alpha_1^{(m-1)}\right)-f_1^{(m-1)}(\xx).$
        Hence, we see that $f_1^{(m-1)}$ is a linear function and furthermore, that $f_1^{(m-2)}$ is a polynomial of at most degree 2. Hence, we see that $f_1$ is a polynomial of degree at most $m.$ We can repeat this process for all of the other $f_i$'s and hence we are done.
    \end{proof}
\end{lem}

The following theorem generalizes Theorem \ref{thm:darm} for independent random vectors. Note that instead of asking coefficients $a_ib_j$ to be zero, here we need matrices $A_iB_j$ to be nonsingular.
\begin{thm}
    Let $\yy_1$ and $\yy_2$ be two independent random vectors in $\RR^k$ defined as $$\yy_1=\sum_{i=1}^m A_i \xx_i \quad \quad \yy_2=\sum_{i=1}^m B_i \xx_i,$$
    where $A_i$, $B_i$ are $k\times k$ matrices and $\{\xx_1,\hdots, \xx_m\}$ is a set of mutually independent vectors in $\RR^n$. If $A_jB_j$ is nonsingular, then the vector $\xx_j$ is Gaussian.
    
    \begin{proof}
        Let $\phi_{x_j}$ be the characteristic function associated to $\xx_j$. Then, $$\prod_{i=1}^m \phi_{\xx_i}(A_i^T \ttt_1+ B_i^T \ttt_2)=\prod_{i=1}^m \phi_{\xx_i}(A_i^T \ttt_1)\phi_{\xx_i}(B_i^T \ttt_2)$$
        where $\ttt_1,\ttt_2 \in \RR^k$.
        \par Assuming $A_i$ is nonsingular, we may take $A_i= I_n$ by transforming $\ttt_1$ appropriately. In this case, we have $$\prod_{i=1}^m \phi_{\xx_i}(\ttt_1+ B_i^T \ttt_2)=\prod_{i=1}^m \phi_{\xx_i}(\ttt_1)\phi_{\xx_i}(B_i^T \ttt_2).$$
        By the properties of characteristic functions, we may find some $c>0$ such that for all $|\ttt_1|<c$ and $|\ttt_2|<c$, we have that $\phi_{\xx_i}(\ttt_1)$, $\phi_{\xx_i}(B_i^T\ttt_2)$, and $\phi_{\xx_i}(\ttt_1+B_i^T\ttt_2)$ do not vanish. Then we find $$\sum_{i=1}^m \psi_{x_i}(s+B_i^T \ttt_2) = \sum_{i=1}^m \psi_{\xx_i}(B_i^T \ttt_2)+\sum_{i=1}^m \psi_{\xx_i}(\ttt_1),$$
        where $\psi_{\xx_i} = \ln |\phi_{\xx_i}|$ and $\psi_{\xx_i}(0)=0.$ We can see that $\psi_{\xx_i}$ are non-negative. Using Lemma \ref{lem:finitediffgen}, we can see that $\sum_{i=1}^m \psi_{\xx_j}(s)$ is a polynomial, say $P(\ttt_1)$, so the characteristic function of $\yy_1$ is of the form $e^{P(\ttt_1)}$. By Lemma \ref{lem:Marc}, $\yy_1$ has an $n$-variate Gaussian distribution. However, $\yy_1$ is a linear combination of independent random vectors, and thus by Lemma \ref{lem:Cramer}, each component of $\yy_1$ is a $k$-variate normal variable.
    \end{proof}
\end{thm}

A peculiarity of the matrix case derives from the distinction between singularity and vanishing of a matrix. In the one-dimensional problem, if one of the coefficients $a_i$ or $b_i$ is zero, the distribution of the corresponding random variable can be completely arbitrary. The same holds for the matrix case if one of the matrices $A_i$ or $B_i$ is zero. However, if a matrix $A_i$ is singular but non-zero, then some linear combinations of elements of the corresponding random vector $\xx_i$ may be normally distributed, but the distribution of $A_i$ is partially arbitrary.

The following corollary gives us the necessary conditions for the ISA problem.
\begin{cor} \label{cor:darmoisGen}
    Let $\yy\in\RR^{d}$ and $\sss\in\RR^{d}$ be two random vectors of the form $$\yy=\begin{pmatrix}
        \yy_1\\
        \yy_2\\
        \vdots \\
        \yy_m
    \end{pmatrix}\quad \quad \sss=\begin{pmatrix}
        \sss_1\\
        \sss_2\\
        \vdots \\
        \sss_m
    \end{pmatrix}$$ where each $\sss_i \in \RR^k$ are mutually independent and $\yy_i\in \RR^k$ are pairwise independent such that $\yy= A\sss$ where $A\in\mathcal{B}_m(k)$ is a block square matrix with blocks $A_{ij}$.
    If $A$ has two non-singular block matrices in the same column $A_{1j},\ldots, A_{mj}$, then $\sss_j$ is either Gaussian or constant.
\end{cor}
\begin{prop}
    Let $\sss\in \RR^{d}$ be a vector with mutually independent non-Gaussian vector components $\sss_i$ (the variables in the components $\sss_i$ do not need to be independent). Let $A$ be an orthogonal $n\times n$ matrix and $\yy=A\sss$ be a random vector.  Then the following statements hold:
    \begin{enumerate}
        \item If the components $\yy_i$ are mutually independent then they are pairwise independent.
        \item If $A\in \mathcal{P}_m O(k)$ is a block orthogonal matrix, then the components $\yy_i$ are mutually independent.
        \item If the components $\yy_i$ are pairwise independent then $A$ has at most one non-singular block at each column.
    \end{enumerate}

    \begin{proof}
    Statement $(1)$ is automatic by definition.
    
    The second one follows because $A \sss$ simply permutes the vectors $\sss_i$ and makes the components of each $\sss_i$ linear combinations of the others. Hence, we see that since $\sss$ is mutually independent, the multiplication by $A$ does not affect the independence of the components.
    
    Lastly, to prove $(3)$ suppose $A$ 
    has some sub-block structure 
        such that none of the sub-blocks is a singular matrix. If this is the case, apply Corollary \ref{cor:darmoisGen} twice, then two components of $\sss$ are Gaussian, which contradicts our hypothesis.
    \end{proof}
\end{prop}

Another peculiarity of the vector case is that no vector component $\sss_i$ can be Gaussian.
Moreover, this proposition shows us that block orthogonal matrices alone are not necessary for independence of $\yy$, and that orthogonal matrices with singular blocks are not sufficient. 
This suggests that there is no direct generalization of Comon's result. The set of orthogonal matrices that preserve the independence of $\yy$ are orthogonal matrices with certain rank restrictions on their blocks. What is clear is that this is not a finite set, so it makes recovery of the original matrix $A$ challenging.

\begin{exmp}
Consider the orthogonal matrix
$Q \in \mathcal{B}_2(2)$ from Example \ref{ex:counterex}, and note that all blocks $B_{ij}$ are singular. Let $x_1$ and $x_2$ be two independent random variables and consider the random vector $\sss=(x_1, -x_1, x_2, -x_2)^T,$ such that the subvectors $\sss_1=(x_1, -x_1)^T$ and $\sss_2=(s_2, -s_2)^T$ are independent. Then 
$$Q\left(\begin{array}{c}
     x_1\\ -x_1\\ x_2\\ -x_2 \\ 
\end{array}\right)=\left(\begin{array}{c}
     -2x_1\\ 2x_1\\ -2x_2\\ 2x_2 \\ 
\end{array}\right)=\yy,$$
and it follows that the random vectors $\yy_1=(-2x_1, 2x_1)^T$ and $\yy_2=(-2x_2, 2x_2)^T$ are also independent.
\end{exmp}

Similarly to Theorem \ref{thm:algPICA}, there is one extra condition on $A$ that can be assumed in order to determine the structure of the matrix $A$. The following theorem follows from \cite{Theis2004}.

\begin{thm}\label{thm:blockortho}
Let $\sss = (\sss_1,\ldots,\sss_m)^T\in\RR^{d}$ be a random vector with partitioned independence and no Gaussian components $\sss_i\in\RR^k$. 
Let $A\in\mathcal{B}_m(k)$ be an invertible matrix with the block structure such that each block $A_{ij}$ is either invertible or zero.
If $\yy= A\sss\in\RR^{d}$ has also partitioned independence, then there is exactly one non-zero block in each column $A_{1j},\ldots,A_{mj}.$ If $\yy$ is white and $A$ is orthogonal, then $A\in\mathcal{P}_m O(n).$
\end{thm} 

\begin{proof}
    Assume that $A$ has a column with two invertible blocks $A_{i_1j}$, $A_{i_2j}$. If we apply Corollary \ref{cor:darmoisGen} to $\yy_{i_1}$ and $\yy_{i_2}$ then $\sss_j$ must be Gaussian, which is a contradiction.
    {If $\yy$ is white and $A$ is orthogonal the result also follows from Theorem \ref{thm:generalThmAlgPICA}. That is, if $\sss$ has partitioned independence, from Proposition \ref{prop:partitioning} it is clear that $\kappa_r(\sss)\in\cV_{\mathbf{I}}$ with $\mathbf{I} = I_1\sqcup\ldots\sqcup I_m$, where $I_i$ contains the indices of $\sss_i$. Then, from Theorem \ref{thm:generalThmAlgPICA} it follows that $A\in\mathcal{P}_m O(n).$ 
    }
\end{proof}

Note that Theorem \ref{thm:generalThmAlgPICA} generalizes the previous result as independence of $\sss$ is not strictly required and the vectors $\sss_i$ are allowed to have different sizes.

{
\begin{cor}
    Let $\sss = (\sss_1,\ldots,\sss_m)^T\in\RR^{d}$ be a random vector with partitioned independence and no Gaussian components. Additionally, within the partitions $\sss_i$, suppose that all components are mutually mean independent of one another.
    Let $A\in\mathcal{B}_m(k)$ be an invertible matrix such that each block $A_{ij}$ is either invertible or zero. Then, if $\yy= A\sss$ is white and $A$ is orthogonal it follows that $A\in\mathcal{P}_m SP(k).$
\begin{proof}
    The result follows directly from Theorem \ref{thm:blockortho} and Corollary \ref{cor:reflInvTensor} since $\kappa_4(\sss)$ is a reflexionally invariant tensor (see Proposition \ref{prop:meanIndepTensor}).
    
\end{proof}
\end{cor}
}

\section*{Acknowledgement}
The authors thank Serkan Hoşten for suggesting the problem and for his guidance and support throughout this work. We also thank Bernd Sturmfels, Piotr Zwiernik and Anna Seigal for their helpful conversations.
Part of this research was done while MG-L was visiting the Institute for Mathematical and Statistical Innovation (IMSI), which is supported by the National Science Foundation (Grant No. DMS-1929348). MS was supported by a Fulbright Open Study/Research Award.

\bibliographystyle{plain}
		\bibliography{report}
\end{document}